\documentclass[11pt]{amsart}
\usepackage[top=1.0in,bottom=1.0in,left=1.0in,right=1.0in]{geometry}
\usepackage[T1]{fontenc}
\usepackage[utf8]{inputenc}
\usepackage{lmodern}
\usepackage{amsmath,amsthm,amssymb,mathtools}
\usepackage{enumitem}
\usepackage{xcolor}
\usepackage{tikz}
\usetikzlibrary{arrows.meta}
\usepackage{hyperref}
\usepackage{microtype}
\definecolor{ash}{RGB}{120,135,125}
\definecolor{figBlue}{HTML}{35689A}
\definecolor{figTeal}{HTML}{24877D}
\definecolor{figOrange}{HTML}{C98232}
\definecolor{figInk}{HTML}{344256}

\tikzset{
  every picture/.append style={line cap=round,line join=round},
  paper vertex/.style={circle,draw=figInk!75,fill=white,line width=.85pt,
    minimum size=5.8mm,inner sep=1.4pt,font=\small,text=figInk},
  paper arc/.style={-{Stealth[length=2.1mm,width=1.5mm]},draw=figInk!65,
    line width=.85pt,shorten <=.6pt,shorten >=.6pt},
  paper cop/.style={paper vertex,draw=figOrange!85!black,fill=figOrange!24,line width=1pt},
  paper reach/.style={paper vertex,draw=figTeal,fill=figTeal!13},
  paper label/.style={font=\small,text=figInk,align=center},
  paper panel/.style={rounded corners=8pt,line width=.6pt}
}

\hypersetup{colorlinks=true, linkcolor=blue, citecolor=blue, urlcolor=ash,
  pdftitle={Directed Isoperimetry and Infinitely Fast Robbers},
  pdfauthor={Arindam Biswas}}
\newtheorem{theorem}{Theorem}[section]
\newtheorem{lemma}[theorem]{Lemma}
\newtheorem{proposition}[theorem]{Proposition}
\newtheorem{corollary}[theorem]{Corollary}
\theoremstyle{remark}
\newtheorem{remark}[theorem]{Remark}
\DeclareMathOperator{\Cay}{Cay}
\DeclareMathOperator{\Sch}{Sch}
\DeclareMathOperator{\diam}{diam}
\DeclareMathOperator{\tw}{tw}
\DeclareMathOperator{\pw}{pw}
\newcommand{\Z}{\mathbb Z}
\newcommand{\R}{\mathbb R}
\newcommand{\F}{\mathbb F}
\newcommand{\normtor}[1]{\left\lVert#1\right\rVert_{\R/\Z}}
\title{Directed isoperimetry and infinitely fast robbers}

\author{Arindam Biswas}
\email{arin.math@gmail.com}
\date{\today}



\allowdisplaybreaks
\begin{document}

\begin{abstract}
We study the fast version of the Cops and Robber game on finite digraphs with unit-speed cops and an unbounded-speed robber. We prove a lower bound for the cop number of every
strongly connected digraph in terms of its order, maximum out-degree, and
out-vertex expansion, extending Mehrabian's isoperimetric lower bound
to digraphs, where connected components are replaced
by forward reachable sets. For strongly connected abelian Cayley digraphs of fixed
out-degree, we show that the ratio of the cop number to the number of vertices is comparable to each of
out-vertex expansion, vertex expansion of the underlying undirected graph,
minimum nontrivial character displacement, the square root of the
unnormalized Laplacian spectral gap of the underlying undirected graph, and
the reciprocal of its diameter. The
comparison constants depend polynomially on the out-degree. In the
undirected case, we also compare pathwidth with these parameters,
complementing the bounded-degree treewidth comparison of
Alon and Mehrabian. Combining the directed lower bound with recent
cyclic sweep techniques for abelian Cayley graphs, we determine the order of
the cop number of positively oriented tori and directed Heisenberg graphs.
\end{abstract}

\maketitle

\section{Introduction}

In the classical game of Cops and Robber on a finite connected graph $\Gamma$,
the cops choose their initial vertices, and the robber then chooses its
initial vertex. The cops and robber move alternately, with the cops moving
first; each cop or robber may stay at its current vertex or move to
an adjacent vertex on its turn. The cops win when a cop occupies the robber's vertex.
The cop number $c(\Gamma)$ is the least number of cops that can guarantee
capture. The game was introduced independently by Nowakowski--Winkler
\cite{NowakowskiWinkler} and Quilliot \cite{Quilliot}; see the book of
Bonato and Nowakowski \cite{BonatoNowakowski} for background.

In the fast-robber variant, each cop still moves at most one edge per turn,
while the robber may traverse a path of length at most $s\geq1$, avoiding all
vertices occupied by cops. We study the unbounded-speed case, in which the
robber may traverse any finite cop-free path, and write
$c_{1,\infty}(\Gamma)$ for the least number of cops that guarantee capture.
On a digraph, both the cops and the robber follow arc directions. The
fast-robber variant was developed systematically by
Fomin--Golovach--Kratochv\'il--Nisse--Suchan \cite{FominEtAl} and has since
been studied in several settings; see, for example,
\cite{AlonMehrabian2011,AlonMehrabian2015,
FriezeKrivelevichLoh,
KinnersleyTownsend,Mehrabian2012}.
The classical directed game, in which the robber also has unit speed, has been
studied by Frieze--Krivelevich--Loh, Loh--Oh, Kinnersley, and
Hosseini--Mohar, among others
\cite{FriezeKrivelevichLoh,LohOh,KinnersleyLength,HosseiniThesis,HosseiniMohar}.
Other directed graph-searching games also use a fast robber, but allow cops
to occupy sets of vertices in one move; they therefore define a different
pursuit game \cite{BerwangerDawarHunterKreutzer}. The game studied here
combines features of these lines of work: arc directions constrain both
players, as in the classical directed game, and the robber has unbounded
speed, while each cop still moves along a single arc. The cyclic-sweep method
of \cite[Theorem~1.3 and Proposition~4.6]{Biswas2026} gives upper bounds
for the cop number of abelian Cayley graphs and digraphs.

Mehrabian related the unbounded-speed cop number of an undirected graph to
its vertex expansion \cite[Theorem~3.3(c)]{Mehrabian2012}. For every
connected undirected graph $\Gamma$ of order $n$ and maximum degree $\Delta$,
he proved
\begin{equation}\label{eq:mehrabian-lower-intro}
  c_{1,\infty}\left(\Gamma\right)\geq
  \frac{n\,\iota_v\left(\Gamma\right)}{4\Delta},
\end{equation}
where $\iota_v(\Gamma)$ is the vertex-isoperimetric number. For bounded
degree, this gives a lower bound of order $n\iota_v(\Gamma)$, whereas
cyclic sweeps give upper bounds in terms of group quotients
\cite[Proposition~4.6]{Biswas2026}. This raises two questions: whether
the quotient upper bound can be controlled by vertex expansion, and whether
an analogous lower bound holds for digraphs.
The upper-bound question for undirected abelian Cayley graphs was posed in
\cite[Question~5.5]{Biswas2026}.

We prove a lower bound in terms of out-vertex expansion for every strongly
connected digraph and a matching upper bound, up to degree-dependent
constants, for abelian Cayley digraphs. The lower-bound proof has two
elements, a robber invariant and an isoperimetric escape proposition, in which
forward reachable sets play the role of connected components; see
Section~\ref{sec:directed-lower}.
For the upper bound, we relate the character displacement governing a
cyclic sweep \cite[Proposition~4.6]{Biswas2026} to vertex expansion through
the spectral gap of the underlying undirected graph.

\subsection{Main results}
Our first result is the following lower bound.

\begin{theorem}\label{thm:intro-directed-lower}
Let $\Gamma^+$ be a strongly connected digraph on $n\geq3$ vertices, with
maximum out-degree $D^+$. Then
\[
  c_{1,\infty}\left(\Gamma^+\right)>
  \frac{n\,\iota_v^+\left(\Gamma^+\right)}{4\left(D^++1\right)}.
\]
\end{theorem}

The proof is given in Theorem~\ref{thm:directed-lower}. Together with the trivial upper bound $c_{1,\infty}\left(\Gamma^+\right)\leq n$,
Theorem~\ref{thm:intro-directed-lower} implies that every family of
strongly connected digraphs with out-degree at most $D$ and out-vertex
expansion at least $\varepsilon>0$ has
$c_{1,\infty}\left(\Gamma^+\right)=\Theta_{D,\varepsilon}\left(n\right)$
(Corollary~\ref{cor:directed-expanders}). Combining the directed lower bound with cyclic quotient sweeps
\cite[Proposition~4.6]{Biswas2026} gives the following two results. The torus $\vec T_{m,k}$ and the Heisenberg graph
$\mathcal H_p$ are defined in Sections~\ref{sec:directed-tori}
and~\ref{sec:heisenberg}, respectively.

\begin{theorem}[Directed tori and Heisenberg graphs]
\label{thm:intro-applications}
\leavevmode
\begin{enumerate}[label=\textup{(\roman*)}]
\item For integers $m\geq3$ and $k\geq1$, the positively oriented torus
$\vec T_{m,k}$ has order
$n=m^k$. Uniformly in $m$ for fixed $k$,
\[
 c_{1,\infty}\left(\vec T_{m,k}\right)
 =\Theta_k\left(m^{k-1}\right)
 =\Theta_k\left(n^{1-1/k}\right).
\]
\item For every odd prime $p$, the directed Heisenberg graph
$\mathcal H_p$ satisfies
\[
 c_{1,\infty}\left(\mathcal H_p\right)
 =\Theta\left(p^2\right).
\]
\end{enumerate}
\end{theorem}

Explicit bounds and proofs are given in Theorems~\ref{thm:directed-tori}
and~\ref{thm:heisenberg-cops}. The torus family has out-degree $D=k$ and
attains the exponent $n^{1-1/D}$ in the upper bound for families whose
out-degree is bounded by the fixed number $D$ in
\cite[Theorem~1.3]{Biswas2026}. The
Heisenberg argument also
extends to composite moduli and higher dimension; see
Theorem~\ref{thm:heisenberg-general}. Since $\mathcal H_p$ has out-degree $2$
and cop number of order $n^{2/3}$, it also shows that the abelian
$O_2(n^{1/2})$ bound of \cite[Theorem~1.3]{Biswas2026} cannot extend to
all two-generator finite groups.

The directed lower bound, cyclic sweeps
\cite[Proposition~4.6]{Biswas2026}, and the Buser inequality
\cite{KlartagKozmaRalliTetali} give the following comparisons for abelian
Cayley digraphs.

\begin{theorem}\label{thm:intro-characterization}
Let $\Gamma^+=\Cay\left(G,S\right)$ be a strongly connected abelian Cayley digraph
on $n\geq3$ vertices and with out-degree $d$. Let
$\Gamma:=U(\Gamma^+)=\Cay\left(G,S\cup(-S)\right)$ be its underlying
undirected graph, and let $\lambda_1\left(\Gamma\right)$ denote the
smallest positive eigenvalue of its unnormalized Laplacian
$L_{\Gamma}=D_{\Gamma}-A_{\Gamma}$, where
$D_{\Gamma}$ and $A_{\Gamma}$ are its degree and adjacency
matrices, respectively.
Then
\[
 \begin{split}
 c_{1,\infty}\left(\Gamma^+\right)
 &=\Theta_d\left(n\iota_v^+\left(\Gamma^+\right)\right)
 =\Theta_d\left(n\iota_v\left(\Gamma\right)\right)
 =\Theta_d\left(n\delta\left(\Gamma^+\right)\right)
 =\Theta_d\left(n\sqrt{\lambda_1\left(\Gamma\right)}\right)
 =\Theta_d\left(\frac{n}{\diam\left(\Gamma\right)}\right).
 \end{split}
\]
\end{theorem}

The spectral comparison uses the Buser inequality of
Klartag--Kozma--Ralli--Tetali \cite{KlartagKozmaRalliTetali}; the diameter
comparison uses the Babai--Szegedy bound stated in
\cite[Proposition~2.2]{Skresanov} and ball doubling
(Lemma~\ref{lem:ball-doubling}). The qualitative
expansion--diameter comparison also follows from
\cite[Proposition~4.7 and Corollary~1.6]{BreuillardTointon}.
The undirected case answers \cite[Question~5.5]{Biswas2026}.

The comparison constants above depend polynomially on $d$. In particular,
the ball-doubling argument of Lemma~\ref{lem:ball-doubling}, combined with
the elementary bound $\iota_v\left(\Gamma\right)\leq2$ of
\eqref{eq:hplus-at-most-two}, gives
$\iota_v\left(\Gamma\right)\leq24d/\diam\left(\Gamma\right)$
(Proposition~\ref{prop:expansion-diameter}).
For undirected abelian Cayley graphs of fixed degree, the quantities in
Theorem~\ref{thm:intro-characterization} are also comparable
to $\tw\left(\Gamma\right)+1$ and $\pw\left(\Gamma\right)+1$, where
$\tw\left(\Gamma\right)$ and $\pw\left(\Gamma\right)$ denote the treewidth
and pathwidth of $\Gamma$, respectively, as defined in
Section~\ref{sec:notation}. The comparison of the cop number with
treewidth already holds for every connected graph of bounded degree: the
two quantities agree up to a factor $\Delta+1$
\cite[Theorem~3.3]{AlonMehrabian2015}. Since $\tw\leq\pw$, this also
gives $c_{1,\infty}\left(\Gamma\right)\leq\pw\left(\Gamma\right)+1$.
What we add for undirected abelian Cayley graphs is the reverse estimate
$\pw\left(\Gamma\right)+1\leq3n\delta\left(\Gamma\right)$
(Proposition~\ref{prop:pathwidth-quotient}), obtained by ordering the
quotient residues linearly and retaining a fixed block of fibres in every
bag of the path decomposition; see Section~\ref{sec:width}.

\section{Definitions, notation, and preliminaries}
\label{sec:notation}

\subsection{General conventions}

We use $\Z_m=\Z/m\Z$ and $\F_p=\Z/p\Z$ for a prime $p$.
Write $e_i$ for the $i$th coordinate vector. We use $a\cdot b=\sum_i a_i b_i$ over the
coefficient ring in question and $\|z\|_1=\sum_i|z_i|$ for $z\in\R^d$.

For positive quantities $f$ and $g$, $f=\Theta_q\left(g\right)$ means that
$c_qg\leq f\leq C_qg$ for positive constants
$c_q,C_q$ depending only on the displayed parameter $q$, uniformly in
all other parameters.  Without a subscript the constants are absolute.
The notation $f=O_q\left(g\right)$ means $f\leq C_qg$, and $f=\Omega_q\left(g\right)$ means
$f\geq c_qg$, with the same convention on constants.

We use multiplicative notation for general groups and additive notation
for abelian groups, with identities $1$ and $0$, respectively.
Write $\langle S\rangle$ for the subgroup generated by $S$ and
$S^{-1}=\{s^{-1}:s\in S\}$, or $-S$ in additive notation.
For $K\leq G$, write $K\backslash G=\{Kx:x\in G\}$ for the right-coset
space; when $K\trianglelefteq G$, the quotient group is also denoted $G/K$.
A fibre of a homomorphism $\pi$ over $y$ is $\pi^{-1}(\{y\})$;
occupying a full fibre means placing a cop at each of its vertices.
We use the commutator convention $[u,v]=u^{-1}v^{-1}uv$ and write
$G/[G,G]$ for the abelianization.

For a finite abelian group $G$, an additive character is a homomorphism
$\chi:G\to\R/\Z$, and
$\widehat G=\operatorname{Hom}\left(G,\R/\Z\right)$ is
the group of these characters under pointwise addition. We write
$\normtor{x}$ for the distance of a real representative of
$x\in\R/\Z$ to the nearest integer.
The corresponding complex-valued character is
$g\mapsto\exp\left(2\pi i\chi\left(g\right)\right)$.

\subsection{Graphs and the pursuit game}
\label{sec:graphs-game}

We write $\Gamma$ for an undirected graph and $\Gamma^+$ for a digraph, with
indexed versions such as $\Gamma_j$ and $\Gamma_j^+$ for families.
For a digraph $\Gamma^+$, its underlying undirected graph $U(\Gamma^+)$ is
the simple graph in which distinct vertices are adjacent whenever at least
one of the two possible arcs joins them in $\Gamma^+$. When a digraph and
its underlying undirected graph occur together, we write
$\Gamma=U(\Gamma^+)$.

All digraphs in this paper are finite and loopless, with no parallel arcs,
unless explicitly stated otherwise. Write $V\left(\Gamma^+\right)$ for the vertex set,
$n=|V\left(\Gamma^+\right)|$ for its order, and $u\to v$ for an arc from $u$ to $v$.
The out-degree $d^+\left(v\right)$ counts arcs leaving $v$, and the in-degree $d^-\left(v\right)$
counts arcs entering $v$; $D^+=\max_v d^+\left(v\right)$ and
$D^-=\max_v d^-\left(v\right)$.  In an undirected graph, $\Delta$ denotes the maximum
degree.  A digraph is strongly connected if every ordered pair of vertices
is joined by a directed path.  It is Eulerian here if
$d^+\left(v\right)=d^-\left(v\right)$ at every vertex.  It is vertex-transitive if, for every
pair of vertices, an automorphism sends the first to the second.
The induced digraph $\Gamma^+-X$ is obtained by deleting the vertices of $X$
and all incident arcs.  The reverse digraph $\left(\Gamma^+\right)^{\mathrm{rev}}$ is
obtained by reversing every arc.
An undirected graph is
identified with the symmetric digraph obtained by replacing every edge with
the two corresponding arcs.  Players are permitted to stay at their current
vertices.  Multiple cops may occupy the same vertex.  The cops know the
robber's current vertex at every stage.  A directed path has no repeated vertices; a walk may
have repeated vertices.  Any finite directed walk can be reduced to a path
with the same endpoints by deleting closed subwalks.  We call a digraph
oriented if no pair of vertices is joined in both directions.

In the $\left(1,\infty\right)$ game, the cops choose their initial vertices, after which
the robber chooses its initial vertex.  In each round all cops first stay put
or traverse one arc.  If a cop then occupies the robber's vertex, the robber is
captured.  Otherwise the robber may stay put or traverse any finite directed
path containing no vertex currently occupied by a cop.  The least number of
cops having a winning strategy is denoted by $c_{1,\infty}\left(\Gamma^+\right)$.
For an integer $s\geq1$, $c_{1,s}\left(\Gamma^+\right)$ uses the same rules with the
robber's path restricted to length at most $s$.  Path length counts arcs
(or edges in an undirected graph), and a cop-free path avoids every vertex
currently occupied by a cop, including its endpoints.

An undirected graph is connected if every pair of vertices is joined by a
path. In a connected undirected graph $\Gamma$,
$\operatorname{dist}_{\Gamma}\left(u,v\right)$ is the
minimum length of a path from $u$ to $v$, and
$\diam\left(\Gamma\right)=\max_{u,v}\operatorname{dist}_{\Gamma}\left(u,v\right)$.
The radius-$r$ ball about $v$ is the set of vertices at distance at most
$r$ from $v$.  The cycle on $m$ vertices is denoted by $C_m$.
A doubling constant for a graph is a constant $C$ such that every integer-radius
ball satisfies $|B(2r)|\leq C|B(r)|$ for $r\geq1$.
The Cartesian product $\Gamma_1\square\Gamma_2$ has vertex set $V\left(\Gamma_1\right)\times V\left(\Gamma_2\right)$;
two vertices are adjacent when one coordinate agrees and the other
coordinates are adjacent in the corresponding factor.
Thus $C_m^{\square k}$ is the Cartesian product of $k$ copies of $C_m$.

A tree decomposition \cite{RobertsonSeymour1986} of an undirected graph $\Gamma$ consists of a finite tree
$\mathcal T$ and a set $X_t\subseteq V(\Gamma)$ for each node $t\in V(\mathcal T)$.
Each set $X_t$ is called a bag. The bags satisfy the following conditions:
\begin{enumerate}[label=\textup{(\roman*)}]
\item Every vertex of $\Gamma$ belongs to at least one bag.
\item The endpoints of every edge of $\Gamma$ belong together to some bag.
\item For each vertex $v\in V(\Gamma)$, the nodes $t$ whose bags contain $v$
form a connected subtree of $\mathcal T$.
\end{enumerate}
The width of a
tree decomposition is $\max_t|X_t|-1$, and the minimum such width is the
treewidth $\tw(\Gamma)$. A path decomposition is a tree decomposition whose
indexing tree is a path; its minimum width is the pathwidth $\pw(\Gamma)$. Thus
$\tw(\Gamma)\leq\pw(\Gamma)$.

\subsection{Cayley and Schreier digraphs}
\label{sec:cayley-notation}

For a finite group $G$ and a set $S\subseteq G\setminus\{1\}$, the
right-Cayley digraph $\Cay\left(G,S\right)$ has vertex set $G$ and an arc
$g\to gs$ for each $g\in G$ and $s\in S$.  The elements of $S$ are called
generators. Since $G$ is finite, this digraph is strongly connected exactly when $S$ generates $G$.
Its underlying undirected graph is
$\Cay\left(G,(S\cup S^{-1})\setminus\{1\}\right)$.
A cyclic
quotient is a surjective homomorphism from a group onto a cyclic group.
A cyclic sweep \cite[Proposition~4.6]{Biswas2026} moves cops around
the quotient cycle while placing cops on every
vertex in the fibres over selected cycle positions.

For a subgroup $K\leq G$, the right-coset Schreier digraph
$\Sch\left(K\backslash G,S\right)$ has vertex set $K\backslash G$ and an
arc $Kx\to Kxs$ for every $x\in G$ and $s\in S$. Loops and parallel arcs
may be suppressed, since they do not change legal moves; $K$ need not be
normal.

For a strongly connected abelian Cayley digraph $\Gamma^+=\Cay\left(G,S\right)$
with $G\neq\{0\}$, we define its character-displacement parameter by
\begin{equation}\label{eq:def-delta}
 \delta\left(\Gamma^+\right):=
 \min_{\substack{\chi\in\widehat G\\ \chi\neq0}}
 \max_{s\in S}\normtor{\chi\left(s\right)}.
\end{equation}
The finite nontrivial abelian group $G$ has a nonempty finite set of
nontrivial characters, so the minimum is attained.  Each such character
is nonzero on some element of $S$, since $S$ generates $G$; hence this
parameter is positive. The same formula defines $\delta\left(\Gamma\right)$
for a connected undirected abelian Cayley graph $\Gamma=\Cay\left(G,S\right)$
with $S=-S$.

For a strongly connected abelian Cayley digraph $\Gamma^+=\Cay(G,S)$, we write
\[
 T:=S\cup(-S),\qquad
 \Gamma:=U(\Gamma^+)=\Cay(G,T),\qquad
 \widehat d:=|T|
\]
for the symmetric generating set, the underlying undirected graph, and its
degree. Every character $\chi$ satisfies
$\normtor{\chi(-s)}=\normtor{\chi(s)}$, and hence
$\delta\left(\Gamma\right)=\delta\left(\Gamma^+\right)$.
When convenient, we abbreviate $h^+:=\iota_v^+(\Gamma^+)$ and
$\widehat h:=\iota_v(\Gamma)$; these isoperimetric numbers are defined in
\eqref{eq:def-hplus} below.

For $a\in\Z_m$, define its cyclic displacement by
\[
 |a|_m:=\min_{j\in\Z}|\widetilde a+jm|,
\]
where $\widetilde a$ is any integer representative of $a$.  For
$A\subseteq\Z_m$, put $Q(A):=\max_{a\in A}|a|_m$; for a homomorphism
$\pi:G\to\Z_m$, put $Q_\pi(S):=\max_{s\in S}|\pi(s)|_m$.

\subsection{Neighbourhoods, boundaries, and reachability}

For $A\subseteq V\left(\Gamma^+\right)$, define
\[
 N^+\left(A\right):=\{v:\text{there are }u\in A\text{ and an arc }u\to v\},
 \qquad N^+[A]:=A\cup N^+\left(A\right),
\]
and
\[
 \partial_v^+A:=N^+\left(A\right)\setminus A.
\]
Define also
\[
 N^-\left(A\right):=\{v:\text{there are }u\in A\text{ and an arc }v\to u\},
 \qquad
 \partial_v^-A:=N^-\left(A\right)\setminus A.
\]
The out-vertex-isoperimetric number is
\begin{equation}\label{eq:def-hplus}
 \iota_v^+\left(\Gamma^+\right):=
 \min_{0<|A|\leq |V\left(\Gamma^+\right)|/2}
 \frac{|\partial_v^+A|}{|A|}.
\end{equation}
For an undirected graph $\Gamma$, we use $\partial_v A=N\left(A\right)\setminus A$ and
write $\iota_v\left(\Gamma\right)$ for the same minimum with $\partial_v A$ in place of
$\partial_v^+A$.
Here $N\left(A\right)$ consists of the vertices adjacent to at least one vertex of $A$.
For disjoint vertex sets $A,B$, $E\left(A,B\right)$ is the set of undirected edges
with one endpoint in each set; each such edge is counted once.
For vertex sets $X,Y$ in a digraph, $e(X,Y)$ denotes the number of arcs with
tail in $X$ and head in $Y$.
An out-closed set has no arc to its complement in the ambient digraph.
The in-vertex-isoperimetric number is
\[
 \iota_v^-\left(\Gamma^+\right):=
 \iota_v^+\left(\left(\Gamma^+\right)^{\mathrm{rev}}\right).
\]
If $\Gamma^+$ is strongly connected, then $\iota_v^+\left(\Gamma^+\right)>0$.  Moreover, if
$n\geq3$, choosing any set of $\lfloor n/2\rfloor$ vertices gives
\begin{equation}\label{eq:hplus-at-most-two}
\iota_v^+\left(\Gamma^+\right)\leq
\frac{\lceil n/2\rceil}{\lfloor n/2\rfloor}\leq2.
\end{equation}
More precisely,
\[
 \frac{\lceil n/2\rceil}{\lfloor n/2\rfloor}
 =\begin{cases}
  1,&n\text{ even},\\[2pt]
  1+\dfrac{2}{n-1},&n\text{ odd}.
 \end{cases}
\]
Thus the bound $2$ is a uniform estimate for all $n\geq3$, attained by
this ratio only at $n=3$; the ratio tends to $1$ as $n\to\infty$.
For a set $A$ of size $\lfloor n/2\rfloor$, the first inequality in
\eqref{eq:hplus-at-most-two} follows from
$\partial_v^+A\subseteq V(\Gamma^+)\setminus A$, whose size is $\lceil n/2\rceil$.
Expansion bounded away from zero means that the relevant isoperimetric
numbers have a common positive lower bound. A directed expander family here
is a family with bounded maximum out-degree and out-vertex-isoperimetric
numbers bounded away from zero.

For $X\subseteq V\left(\Gamma^+\right)$ and $x\notin X$, let
\[
 R\left(x,X\right):=\{y:\text{$y$ is reachable from $x$ by a directed path in
 $\Gamma^+-X$}\}.
\]
The path of length zero is allowed, so $x\in R\left(x,X\right)$.  This is the set of possible destinations of an unbounded-speed robber
at $x$ when the cops occupy $X$; see Figure~\ref{fig:reachable-set}.

\begin{lemma}[Out-closure]\label{lem:out-closure}
For every $X\subseteq V\left(\Gamma^+\right)$ and $x\notin X$,
\[
 \partial_v^+R\left(x,X\right)\subseteq X.
\]
More generally, an arbitrary union of sets of the form $R\left(x,X\right)$ is out-closed
in $\Gamma^+-X$ and consequently has out-boundary contained in $X$.
\end{lemma}

\begin{proof}
Let $u\in R\left(x,X\right)$ and suppose $u\to w$ is an arc with $w\notin X$.  Appending
this arc to an $x$--$u$ path in $\Gamma^+-X$, and deleting any closed
subwalk, shows that
$w\in R\left(x,X\right)$.  The same argument applies when $u$ belongs to a union of
reachable sets.
\end{proof}

\subsection{Spectral and edge-isoperimetric notation}

For an undirected graph $\Gamma$, let $A_{\Gamma}$ be its adjacency matrix, with entry
$1$ at adjacent vertex pairs and $0$ otherwise, and let $D_{\Gamma}$ be the diagonal
matrix of vertex degrees. Its unnormalized Laplacian is
$L_{\Gamma}:=D_{\Gamma}-A_{\Gamma}$. For connected $\Gamma$ with at least two vertices, write
$\lambda_1(\Gamma)$ for the smallest positive eigenvalue of $L_{\Gamma}$, also called
the unnormalized spectral gap.

The unnormalized edge-isoperimetric number of $\Gamma$ is
\[
 h_E(\Gamma):=\min_{0<|A|\leq |V(\Gamma)|/2}
 \frac{|E(A,V(\Gamma)\setminus A)|}{|A|},
\]
where crossing edges are counted once.

\subsection{A classical expansion estimate}

We shall use the following classical estimate to obtain vertex-expansion
lower bounds from diameter. It applies to both abelian and nonabelian
Cayley graphs, since every Cayley graph is vertex-transitive.

\begin{lemma}[Babai--Szegedy local expansion]
\label{lem:babai-szegedy}
Let $\Gamma$ be a finite connected undirected vertex-transitive graph of order $N$
with $N\geq2$ and diameter $R$.  For every nonempty $A\subseteq V\left(\Gamma\right)$,
\begin{equation}\label{eq:babai-szegedy-full}
 \frac{|\partial_v A|}{|A|}
 \geq
 \frac{2\left(1-|A|/N\right)}{R+|A|/N}.
\end{equation}
In particular,
\begin{equation}\label{eq:babai-szegedy-half}
 \iota_v\left(\Gamma\right)\geq\frac{2}{2R+1}.
\end{equation}
\end{lemma}

\begin{proof}
Babai--Szegedy proved the vertex-transitive local-expansion theorem in
\cite[Theorem~2.2]{BabaiSzegedy}.  Formula
\eqref{eq:babai-szegedy-full}, and hence its specialization
\eqref{eq:babai-szegedy-half} for $|A|\leq N/2$, is the quantitative form
stated in \cite[Proposition~2.2]{Skresanov}.
\end{proof}

\section{A directed isoperimetric lower bound}\label{sec:directed-lower}

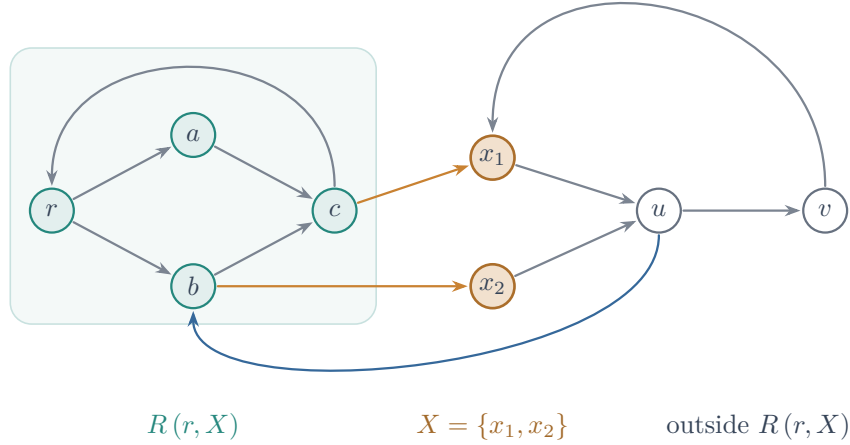
\begin{figure}[tbp]
\centering

\begin{tikzpicture}[x=1.1cm,y=1cm]
\filldraw[paper panel,fill=figTeal!5,draw=figTeal!25] (-.5,-1.5) rectangle (3.9,2.15);
\node[paper reach] (r) at (0,0) {$r$};
\node[paper reach] (a) at (1.7,1) {$a$};
\node[paper reach] (b) at (1.7,-1) {$b$};
\node[paper reach] (c) at (3.4,0) {$c$};
\node[paper cop] (x) at (5.3,.7) {$x_1$};
\node[paper cop] (y) at (5.3,-1) {$x_2$};
\node[paper vertex] (u) at (7.3,0) {$u$};
\node[paper vertex] (v) at (9.3,0) {$v$};
\foreach \a/\b in {r/a,r/b,a/c,b/c,u/v,x/u,y/u}
 \draw[paper arc] (\a) -- (\b);
\draw[paper arc] (c.north) .. controls +(0,2.1) and +(0,2.1) .. (r.north);
\draw[paper arc] (v.north) .. controls +(0,2.7) and +(0,2.7) .. (x.north);
\draw[paper arc,figOrange] (c) -- (x);
\draw[paper arc,figOrange] (b) -- (y);
\draw[paper arc,figBlue] (u.south) .. controls +(0,-1.8) and +(0,-1.4) .. (b.south);
\node[paper label,text=figTeal] at (1.7,-2.85) {$R\left(r,X\right)$};
\node[paper label,text=figOrange!80!black] at (5.3,-2.85) {$X=\{x_1,x_2\}$};
\node[paper label] at (8.5,-2.85) {outside $R\left(r,X\right)$};
\end{tikzpicture}

\caption{Forward reachability in $\Gamma^+-X$. The teal set $R(r,X)$ has out-boundary contained in the orange set $X=\{x_1,x_2\}$. The blue arc enters $R(r,X)$ from outside.}
\label{fig:reachable-set}
\end{figure}

We adapt Mehrabian's large-component strategy
\cite[Lemma~3.1 and Theorem~3.3]{Mehrabian2012}. For an undirected graph,
his proof has two steps. In the pursuit step \cite[Lemma~3.1]{Mehrabian2012},
the cops occupy a set $S$ and the robber stays in a connected component of
$\Gamma-N[S]$ with more than $n/2$ vertices, where $N[S]$ is the closed
neighbourhood of $S$. Such a component is unique.
After the cops move, the robber walks to a vertex in the intersection of the
old and new large components. In the isoperimetric step
\cite[Theorem~3.3]{Mehrabian2012}, a set $S$ with no such component is
considered. A subset-sum argument then selects components whose union has
between $n/4$ and $n/2$ vertices and whose vertex boundary lies in
$N[S]\setminus S$.

Neither step transfers verbatim to digraphs. Forward reachable sets of
distinct vertices may overlap, a large reachable set need not be unique, and
a vertex common to two large reachable sets need not be reachable from the
robber's position. Moreover, the deleted set is the closed out-neighbourhood
$N^+[C]$ of the cop positions $C$. A vertex outside $N^+[C]$ may have an
out-neighbour in $C$, so the out-boundary of a union of reachable sets is
contained in $N^+[C]$ but not necessarily in $N^+[C]\setminus C$. We
therefore make the following adaptation. Proposition~\ref{prop:escape} replaces the
isoperimetric step. It finds a vertex with a large reachable set inside a
prescribed set $Z$, using a greedy union of overlapping reachable sets in
place of the selection of components. In the pursuit step of
Theorem~\ref{thm:directed-lower}, the intersection argument is replaced by
the monotonicity of reachable sets: the robber's next vertex is chosen
inside the set that it can reach at that moment. Bounding the out-boundary
by all of $N^+[C]$ costs the factor $D^++1$ in place of Mehrabian's $\Delta$;
see Remark~\ref{rem:undirected-comparison}.

\begin{proposition}[Escape]\label{prop:escape}
Let $\Gamma^+$ be a strongly connected digraph on $n\geq3$ vertices and put
$h:=\iota_v^+\left(\Gamma^+\right)$.
If $B\subseteq V\left(\Gamma^+\right)$ satisfies $|B|\leq nh/4$ and
$Z\subseteq V\left(\Gamma^+\right)$ satisfies $|Z|>n/2$, then there is a vertex
$v\in Z\setminus B$ for which
\[
 |R\left(v,B\right)|>n/2.
\]
\end{proposition}

\begin{proof}
We put $m=\lfloor n/2\rfloor$ and $b=|B|$. The case $b=0$ follows from
strong connectivity, so we assume $b>0$. Since $h>0$, we may define
$q=\lfloor b/h\rfloor$. We first show that
\begin{equation}\label{eq:escape-integer-bounds}
 b+q\leq m,\qquad 2q\leq m.
\end{equation}
If $h\leq1$, then $b\leq q$ and
$q\leq\lfloor n/4\rfloor$, which gives both inequalities.

Suppose now that $h>1$. By \eqref{eq:hplus-at-most-two} and the case
formula following it, $h\leq1$ when $n$ is even; hence $n$ is odd,
so $n=2m+1$. For every nonempty $A$ with $|A|\leq m$, the definition
\eqref{eq:def-hplus} of $h$ gives $|\partial_v^+A|\geq h|A|>|A|$, and
integrality gives $|\partial_v^+A|\geq|A|+1$. Hence its boundary ratio is at least
$(m+1)/m$, while any set of size $m$ has ratio at most $(m+1)/m$.
Thus $h=(m+1)/m$, and
\[
 b\leq\left\lfloor\frac{(2m+1)(m+1)}{4m}\right\rfloor
 =\left\lceil\frac m2\right\rceil\leq m,
 \qquad
 q=\left\lfloor\frac{bm}{m+1}\right\rfloor=b-1.
\]
The floor identity follows by considering even and odd $m$ separately.
Consequently $b+q=2b-1\leq m$ and $2q=2b-2\leq m$, proving
\eqref{eq:escape-integer-bounds}.

We suppose, for a contradiction, that $|R(v,B)|\leq m$ for every
$v\in Z\setminus B$. This set of starting vertices is nonempty because
$b\leq m<|Z|$. By Lemma~\ref{lem:out-closure}, each reachable set has
out-boundary contained in $B$. Since $0<|R(v,B)|\leq m$, the set $R(v,B)$
is admissible in the definition \eqref{eq:def-hplus} of $h$, so
\[
 h|R(v,B)|\leq|\partial_v^+R(v,B)|\leq b.
\]
It follows that $|R(v,B)|\leq q$. If $q=0$, this contradicts the fact that
$v\in R(v,B)$. Otherwise, we enumerate these reachable sets as
$R_1,\ldots,R_t$. Their union contains $Z\setminus B$ and therefore has
size at least $m+1-b>q$, by \eqref{eq:escape-integer-bounds}.
We take $W$ to be the first partial union whose size exceeds $q$.
The preceding union and the newly added set each have at most $q$ vertices,
so
\[
 q<|W|\leq2q\leq m.
\]
By the second part of Lemma~\ref{lem:out-closure}, the union $W$ is
out-closed in $\Gamma^+-B$, and hence $\partial_v^+W\subseteq B$. Since
$0<|W|\leq m$, the set $W$ is admissible in the definition
\eqref{eq:def-hplus} of $h$, so $h|W|\leq|\partial_v^+W|\leq b$. Hence
$|W|\leq\lfloor b/h\rfloor=q$, a contradiction.
\end{proof}

\begin{theorem}[Directed isoperimetric lower bound]\label{thm:directed-lower}
Let $\Gamma^+$ be a strongly connected digraph on $n\geq3$ vertices.  Write
$D^+$ for its maximum out-degree and $h=\iota_v^+\left(\Gamma^+\right)$.  Then
\begin{equation}\label{eq:refined-directed-lower}
 c_{1,\infty}\left(\Gamma^+\right)>
 \frac{nh}{4\left(D^++1\right)}.
\end{equation}
\end{theorem}

\begin{proof}
If
\[
 \frac{nh}{4\left(D^++1\right)}<1,
\]
then the desired strict inequality follows immediately from
$c_{1,\infty}\left(\Gamma^+\right)\geq1$.  We may therefore assume that
$nh/\left(4\left(D^++1\right)\right)\geq1$.

Let $\kappa$ be a positive integer satisfying
\begin{equation}\label{eq:kappa-hypothesis}
 \kappa\left(D^++1\right)\leq nh/4.
\end{equation}
We describe a strategy by which
the robber evades these $\kappa$ cops forever.

Let $C_t$ be the set of vertices occupied by cops immediately before the cops'
move in round $t+1$, equivalently immediately after the robber's move in round
$t$.  Put
\[
 B_t:=N^+[C_t].
\]
Although several cops may occupy the same vertex, we always have
\begin{equation}\label{eq:B-size}
 |B_t|\leq \left(D^++1\right)|C_t|\leq \kappa\left(D^++1\right)\leq nh/4.
\end{equation}
The robber maintains the following invariant after its move in round $t$
(see Figure~\ref{fig:danger-zone}):
\begin{equation}\label{eq:robber-invariant}
 r_t\notin B_t
 \qquad\text{and}\qquad
 |R\left(r_t,B_t\right)|>n/2.
\end{equation}

For initialization, after the cops choose $C_0$, we apply
Proposition~\ref{prop:escape} with $B=B_0$ and $Z=V\left(\Gamma^+\right)$; its size
hypothesis on $B$ holds by \eqref{eq:B-size}.  It supplies a legal
initial vertex $r_0$ satisfying \eqref{eq:robber-invariant}.

Suppose now that the invariant holds at time $t$.  On their next move the cops
occupy a set $C_{t+1}$ satisfying
\[
 C_{t+1}\subseteq N^+[C_t]=B_t.
\]
Since $r_t\notin B_t$, no cop lands on $r_t$.  Moreover, deleting fewer
vertices can only enlarge a reachable set, and hence
\[
 Z_{t+1}:=R\left(r_t,C_{t+1}\right)
 \supseteq R\left(r_t,B_t\right).
\]
In particular, $|Z_{t+1}|>n/2$.  We define
$B_{t+1}=N^+[C_{t+1}]$ and apply Proposition~\ref{prop:escape} with
$B=B_{t+1}$ and $Z=Z_{t+1}$, again using \eqref{eq:B-size}.  This gives a vertex
$r_{t+1}\in Z_{t+1}\setminus B_{t+1}$ satisfying
\[
 |R\left(r_{t+1},B_{t+1}\right)|>n/2.
\]
Because $r_{t+1}\in R\left(r_t,C_{t+1}\right)$, the robber can legally move from $r_t$ to
$r_{t+1}$.  Thus the invariant is restored and the robber is never captured.
It follows that every positive integer $\kappa$ satisfying
\eqref{eq:kappa-hypothesis} is
insufficient, which proves \eqref{eq:refined-directed-lower}.
\end{proof}

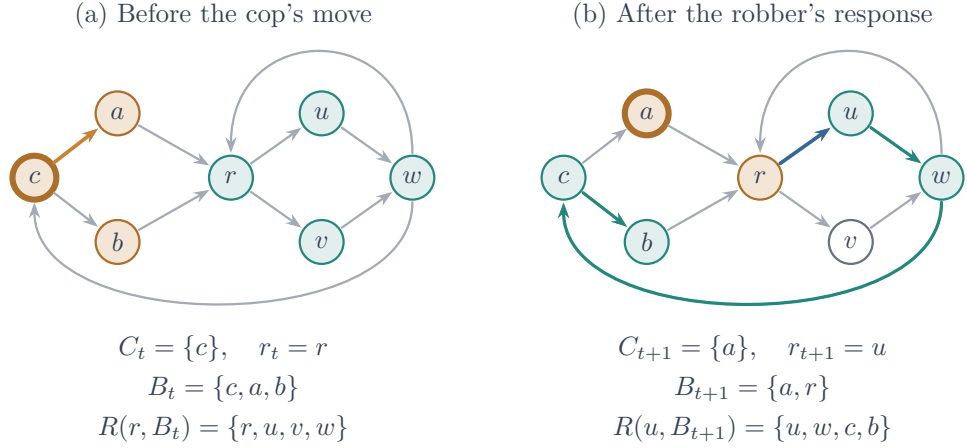
\begin{figure}[tbp]
\centering
\begin{tikzpicture}[x=1cm,y=1cm,
 danger/.style={paper vertex,draw=figOrange!85!black,fill=figOrange!20},
 occupied/.style={danger,double=figOrange!85!black,double distance=.6pt}]
\foreach \stage in {0,1} {
 \begin{scope}[xshift=\stage*7cm]
 \ifnum\stage=0
  \node[paper label] at (2.5,2.15) {(a) Before the cop's move};
  \node[occupied] (c) at (0,0) {$c$};
  \node[danger] (a) at (1.1,.85) {$a$};
  \node[danger] (b) at (1.1,-.85) {$b$};
  \node[paper reach] (r) at (2.6,0) {$r$};
  \node[paper reach] (u) at (3.8,.85) {$u$};
  \node[paper reach] (v) at (3.8,-.85) {$v$};
  \node[paper reach] (w) at (5,0) {$w$};
 \else
  \node[paper label] at (2.5,2.15) {(b) After the robber's response};
  \node[paper reach] (c) at (0,0) {$c$};
  \node[occupied] (a) at (1.1,.85) {$a$};
  \node[paper reach] (b) at (1.1,-.85) {$b$};
  \node[danger] (r) at (2.6,0) {$r$};
  \node[paper reach] (u) at (3.8,.85) {$u$};
  \node[paper vertex] (v) at (3.8,-.85) {$v$};
  \node[paper reach] (w) at (5,0) {$w$};
 \fi
 \foreach \tail/\head in {c/a,c/b,a/r,b/r,r/u,r/v,u/w,v/w}
  \draw[paper arc,draw=figInk!45] (\tail) -- (\head);
 \draw[paper arc,draw=figInk!45] (w.north) .. controls +(0,1.8) and +(0,1.8) .. (r.north);
 \draw[paper arc,draw=figInk!45] (w.south) .. controls +(0,-1.8) and +(0,-1.8) .. (c.south);
 \ifnum\stage=0
  \draw[paper arc,figOrange,line width=1.4pt] (c) -- (a);
  \node[paper label,anchor=north] at (2.5,-1.95)
   {$C_t=\{c\},\quad r_t=r$\\[3pt]
    $B_t=\{c,a,b\}$\\[3pt]
    $R(r,B_t)=\{r,u,v,w\}$};
 \else
  \draw[paper arc,figBlue,line width=1.4pt] (r) -- (u);
  \draw[paper arc,figTeal,line width=1.2pt] (u) -- (w);
  \draw[paper arc,figTeal,line width=1.2pt] (w.south) .. controls +(0,-1.8) and +(0,-1.8) .. (c.south);
  \draw[paper arc,figTeal,line width=1.2pt] (c) -- (b);
  \node[paper label,anchor=north] at (2.5,-1.95)
   {$C_{t+1}=\{a\},\quad r_{t+1}=u$\\[3pt]
    $B_{t+1}=\{a,r\}$\\[3pt]
    $R(u,B_{t+1})=\{u,w,c,b\}$};
 \fi
 \end{scope}
}
\end{tikzpicture}
\caption{One round illustrating the reachability invariant. Double circles
mark the cop; orange vertices form its closed out-neighbourhood, and teal
vertices form the displayed reachable set after deleting that neighbourhood.
The cop moves $c\to a$ (orange arrow), leaving $r$ unoccupied. In
$\Gamma^+-\{a\}$, the set $Z_{t+1}=R(r,\{a\})$ contains all vertices except
$a$, including the previous four-vertex reachable set. The robber moves
$r\to u$ (blue arrow). Although $r\in B_{t+1}$, this move is legal because
only $a$ is occupied. From $u$, the teal path $u\to w\to c\to b$ avoids
$B_{t+1}$, restoring a reachable set of size $4>7/2$.}
\label{fig:danger-zone}
\end{figure}

\begin{corollary}[Directed expanders]\label{cor:directed-expanders}

Let $\left(\Gamma^+_j\right)$ be a family of strongly connected digraphs such that
$|V\left(\Gamma^+_j\right)|\to\infty$, the maximum out-degrees are bounded by a constant
$D$, and $\iota_v^+\left(\Gamma^+_j\right)\geq\varepsilon>0$.  Then
\[
 c_{1,\infty}\left(\Gamma^+_j\right)=\Theta_{D,\varepsilon}\left(|V\left(\Gamma^+_j\right)|\right).
\]
If, instead, $\left(\Gamma^+_j\right)$ is an Eulerian family with maximum in- and
out-degrees at most $D$, it is enough to assume that
\[
 \iota_v\left(\Gamma_j\right)\geq\varepsilon,
 \qquad\text{where }\Gamma_j=U(\Gamma_j^+).
\]
The same conclusion then holds.
\end{corollary}

\begin{proof}
The lower bound is Theorem~\ref{thm:directed-lower}; the upper bound follows by
placing one cop at every vertex.  In the Eulerian case,
Lemma~\ref{lem:in-out-comparison} gives
$\iota_v^+\left(\Gamma^+_j\right)\geq\varepsilon/\left(D+1\right)$.
\end{proof}

\begin{remark}\label{rem:undirected-comparison}
Applied to an undirected graph, viewed as a symmetric digraph,
Theorem~\ref{thm:directed-lower} recovers the isoperimetric lower bound
\eqref{eq:mehrabian-lower-intro} up to the factor $\Delta/(\Delta+1)$,
uniformly in $n$; this factor tends to $1$ as $\Delta\to\infty$. Its
proof does not use connected components, and therefore applies to every
strongly connected digraph.
\end{remark}

\section{Cayley digraphs: underlying graphs and cyclic quotients}

Let $G$ be a finite group and $S\subseteq G\setminus\{1\}$, with notation
as in Section~\ref{sec:cayley-notation}; put $d=|S|$.

\subsection{Comparing directed and undirected vertex expansion}

We recall Babai's Eulerian boundary comparison
\cite[Lemma~3.4]{Babai2006} and its short proof.

\begin{lemma}[Eulerian in/out comparison]
\label{lem:in-out-comparison}
Let $\Gamma^+$ be a finite Eulerian digraph on at least two vertices, that is,
$d^+\left(v\right)=d^-\left(v\right)$ for every $v\in V\left(\Gamma^+\right)$.  Let $D^-$ be its maximum
in-degree and let $\Gamma=U(\Gamma^+)$ be its underlying undirected graph.  Then,
for every $A\subseteq V\left(\Gamma^+\right)$,
\[
 |\partial_v^-A|\leq D^-|\partial_v^+A|
\]
and
\[
 |\partial_vA|
 \leq\left(D^-+1\right)|\partial_v^+A|.
\]
Consequently,
\begin{equation}\label{eq:in-out-comparison}
 \iota_v^+\left(\Gamma^+\right)\leq\iota_v\left(\Gamma\right)
 \leq\left(D^-+1\right)\iota_v^+\left(\Gamma^+\right).
\end{equation}
\end{lemma}

\begin{proof}
Since $\Gamma^+$ is Eulerian, we have
\[
 \begin{split}
 e\left(A,V\setminus A\right)
 &=\sum_{v\in A}d^+\left(v\right)-e\left(A,A\right)\\
 &=\sum_{v\in A}d^-\left(v\right)-e\left(A,A\right)
 =e\left(V\setminus A,A\right).
 \end{split}
\]
Every vertex of $\partial_v^-A$ is the tail of at least one arc entering $A$,
whereas every arc leaving $A$ has its head in $\partial_v^+A$.  Since each
such head has in-degree at most $D^-$,
\[
 |\partial_v^-A|
 \leq e\left(V\setminus A,A\right)
 =e\left(A,V\setminus A\right)
 \leq D^-|\partial_v^+A|.
\]
Furthermore,
\[
 \partial_vA
 =\partial_v^+A\cup\partial_v^-A,
\]
and therefore
\[
 |\partial_vA|
 \leq\left(D^-+1\right)|\partial_v^+A|.
\]
The reverse pointwise containment
$\partial_v^+A\subseteq\partial_vA$ gives the other
inequality.  We complete the proof by minimizing over $0<|A|\leq|V\left(\Gamma^+\right)|/2$.
\end{proof}

\begin{corollary}[Comparison of in- and out-expansion]
\label{cor:eulerian-expansion}
Let $\Gamma^+$ be a finite strongly connected Eulerian digraph on at least
two vertices, with maximum out-degree $D^+$ and maximum in-degree $D^-$.
Then
\[
 \frac{\iota_v^+\left(\Gamma^+\right)}{D^+}
 \leq\iota_v^-\left(\Gamma^+\right)
 \leq D^-\iota_v^+\left(\Gamma^+\right),
\]
and
\[
 \iota_v^-\left(\Gamma^+\right)
 \leq\iota_v\left(\Gamma\right)
 \leq\left(D^++1\right)\iota_v^-\left(\Gamma^+\right).
\]
Together with \eqref{eq:in-out-comparison}, these bounds compare all three
isoperimetric numbers with constants depending only on $D^+$ and $D^-$.
\end{corollary}

\begin{proof}
Minimizing the boundary inequality in Lemma~\ref{lem:in-out-comparison}
gives $\iota_v^-\left(\Gamma^+\right)\leq D^-\iota_v^+\left(\Gamma^+\right)$.
We apply the same lemma to $\left(\Gamma^+\right)^{\mathrm{rev}}$, which is Eulerian,
has maximum in-degree $D^+$, and has the same underlying undirected graph
as $\Gamma^+$.
This gives the other inequalities.
\end{proof}

\begin{remark}\label{rem:cayley-eulerian}
Every finite vertex-transitive digraph is Eulerian.  Indeed, transitivity
makes the in-degree and out-degree constant, and equality of the total
in-degree and total out-degree makes these two constants equal.  In
particular, a Cayley digraph $\Cay\left(G,S\right)$ is Eulerian with
$D^-=D^+=|S|$.
If $G$ is abelian and $|G|\geq2$, inversion $g\mapsto g^{-1}$ is an isomorphism from
$\Gamma^+=\Cay(G,S)$ to $\left(\Gamma^+\right)^{\mathrm{rev}}$. Consequently
$\iota_v^+(\Gamma^+)=\iota_v^-(\Gamma^+)$ and
$c_{1,\infty}(\Gamma^+)=c_{1,\infty}(\left(\Gamma^+\right)^{\mathrm{rev}})$.
\end{remark}

Corollary~\ref{cor:eulerian-expansion} applies to every Eulerian digraph.
For non-Eulerian digraphs, in- and out-expansion can have different orders.

\begin{proposition}[Separation of in- and out-expansion]
\label{prop:expansion-separation}
For every positive integer $n$ divisible by $6$, there is a strongly
connected non-Eulerian digraph $\Gamma^+_n$ on $n$ vertices such that
\[
 \iota_v^+\left(\Gamma^+_n\right)=\frac3n,
 \qquad \iota_v^-\left(\Gamma^+_n\right)=\frac13.
\]
Consequently the ratio of in-expansion to out-expansion is $n/9$, and is unbounded.
Reversing all arcs gives the opposite separation.
\end{proposition}

\begin{proof}
We partition the vertex set as
$V=A\mathbin{\dot\cup}B$ with $|A|=n/3$ and $|B|=2n/3$, and choose
$b_0\in B$.  The digraph $\Gamma^+=\Gamma^+_n$ has both arcs between every
pair of vertices within $A$ and within $B$, every arc from $B$ to $A$, and
an arc $a\to b_0$ for every $a\in A$.
It is strongly connected. If $0<|U|\leq n/2$ and $U$ meets $B$, then
$\partial_v^+U=V\setminus U$, giving boundary ratio at least $1$.
If $U\subseteq A$, its out-boundary is $(A\setminus U)\cup\{b_0\}$,
so its ratio is $(|A|+1-|U|)/|U|$. This is minimized at $U=A$.
Since $n\geq6$, it follows that
\[
 \iota_v^+\left(\Gamma^+\right)=\frac{1}{|A|}=\frac3n.
\]
The digraph is not Eulerian.  Indeed, for every $a\in A$,
\[
 d^+\left(a\right)=|A|,
 \qquad
 d^-\left(a\right)=|A|-1+|B|.
\]
On the other hand, if $0<|U|\leq n/2$ and $U\cap B\neq\varnothing$, then
$B\setminus U\subseteq\partial_v^-U$ and
$|B\setminus U|\geq n/6$, so
$|\partial_v^-U|/|U|\geq1/3$.  If $U\subseteq A$, then
$B\subseteq\partial_v^-U$, giving an even larger ratio. To attain $1/3$, we choose a set
$U\subseteq B\setminus\{b_0\}$ of size $n/2$, which is possible since
$2n/3-1\geq n/2$. Its in-boundary is exactly $B\setminus U$, of size $n/6$.
Thus
\[
 \iota_v^-\left(\Gamma^+\right)=\frac13,
\]
and the ratio $\iota_v^-/\iota_v^+$ equals $n/9$.  Reversing all arcs gives
the opposite separation.

\end{proof}

\subsection{Lifting strategies from Schreier digraphs}

\begin{proposition}[Expansion of Schreier digraphs]
\label{prop:schreier-expansion}
Let $G$ be a finite group, $K<G$ a proper subgroup, and
$S\subseteq G\setminus\{1\}$, with $d=|S|$.
Let $\Sigma^+=\Sch\left(K\backslash G,S\right)$, with loops and parallel
arcs suppressed.  Then
\[
 \iota_v^+\left(\Sigma^+\right)
 \leq\iota_v\left(U(\Sigma^+)\right)
 \leq\left(d+1\right)\iota_v^+\left(\Sigma^+\right).
\]
\end{proposition}

\begin{proof}
We retain one labelled arc $Kx\to Kxs$ for every coset $Kx$ and every
$s\in S$, including loops and multiplicities.  For each fixed $s$,
right multiplication permutes $K\backslash G$, so the labelled
multidigraph is $d$-in-regular and $d$-out-regular.
We can therefore apply the crossing-arc argument of
Lemma~\ref{lem:in-out-comparison}, counting arcs with multiplicity,
to obtain the stated inequalities.  Suppressing loops and
parallel arcs changes none of the vertex boundaries, so the inequalities
also hold for the simple Schreier digraph.  The simple digraph need not
itself be Eulerian; the argument uses equality of in-degree and out-degree
at each vertex of the labelled multidigraph.
\end{proof}

\begin{proposition}[Schreier lifting]\label{prop:quotient-lifting}
Let $G$ be a finite group, $K\leq G$, and $S\subseteq G\setminus\{1\}$.  If
\[
 \Gamma^+=\Cay\left(G,S\right),\qquad
 \Sigma^+=\Sch\left(K\backslash G,S\right),
\]
then
\begin{equation}\label{eq:quotient-lifting}
 c_{1,\infty}\left(\Gamma^+\right)\leq
 |K|\,c_{1,\infty}\left(\Sigma^+\right).
\end{equation}
The same statement holds for every finite robber speed.
\end{proposition}

\begin{proof}
We replace each cop at a Schreier vertex $Kx$ with one cop at every element of the
right coset $Kx$. If the Schreier cop moves from $Kx$ to $Kxs$, we move every cop
in $Kx$ by right multiplication with $s$.  This maps $Kx$ bijectively onto
$Kxs$, so the new coset is again completely occupied.  A stationary Schreier
cop is simulated by stationary cops.

We project the robber's path by $x\mapsto Kx$ and delete consecutive repetitions.
The result is a directed walk; we erase closed subwalks to obtain a path
with the same initial and terminal cosets.
Since every occupied Schreier vertex has its entire coset occupied in the
Cayley digraph, a legal robber path in the Cayley digraph projects to a legal
path in the Schreier digraph.  When a Schreier cop reaches the robber's
projected vertex, the corresponding coset is
completely occupied in the Cayley digraph, so one of its cops occupies the
robber's actual vertex.  Thus the lifted cops implement the Schreier strategy,
using $|K|$ cops for each Schreier cop.  Projection cannot increase path
length, which proves the same claim for every finite speed.
\end{proof}

\begin{remark}[Nonnormal subgroups]\label{rem:nonnormal-lift}
If $K\trianglelefteq G$, then the Schreier digraph is the usual quotient Cayley
digraph, so Proposition~\ref{prop:quotient-lifting} contains normal quotient
lifting as a special case. Let $m\geq3$ be odd and
\[
 D_{2m}=\langle r,s:r^m=s^2=1,\ srs=r^{-1}\rangle,
 \qquad S=\{s,sr\}.
\]
For $K=\langle s\rangle$, identify the coset $Kr^i$ with $i\in\Z_m$.
Right multiplication by $s$ and $sr$ induces the involutions
$i\mapsto-i$ and $i\mapsto1-i$, respectively.  After loops are suppressed,
$\Sch\left(K\backslash D_{2m},S\right)$ is a path.  A single cop that starts
at one endpoint and moves one step towards the other endpoint in each round
captures the robber, since no cop-free path passes the cop's vertex.
Proposition~\ref{prop:quotient-lifting} therefore gives
\[
 c_{1,\infty}\left(\Cay\left(D_{2m},\{s,sr\}\right)\right)\leq2.
\]
Indeed this Cayley graph is the cycle $C_{2m}$, so equality holds: deleting
the vertex of a single cop leaves a path, along which the robber can always
move to a vertex at distance at least $2$ from the cop.  In contrast,
when $m$ is odd the abelianization of $D_{2m}$ is only $\Z_2$, and lifting from
that normal quotient alone gives only an order-$m$ bound.
\end{remark}

\subsection{Sweeping a cyclic quotient}

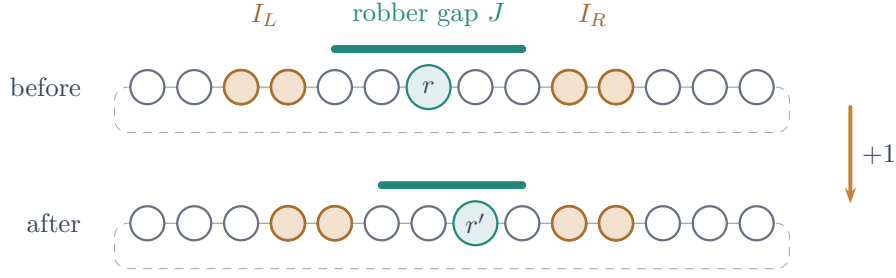
\begin{figure}[tbp]
\centering

\begin{tikzpicture}[x=.62cm,y=1.2cm]
\foreach \row in {0,1} {
 \pgfmathsetmacro{\yy}{-1.5*\row}
 \draw[figInk!55] (0,\yy) -- (13,\yy);
 \draw[figInk!45,dashed,rounded corners=5pt] (0,\yy) -- (-.7,\yy) -- (-.7,\yy-.5) -- (13.7,\yy-.5) -- (13.7,\yy) -- (13,\yy);
 \foreach \i in {0,...,13} \node[paper vertex,minimum size=4.5mm] (v\row\i) at (\i,\yy) {};
 \foreach \i in {9,10} \node[paper cop,minimum size=4.5mm] at (v\row\i) {};
}
\foreach \i in {2,3} \node[paper cop,minimum size=4.5mm] at (v0\i) {};
\foreach \i in {3,4} \node[paper cop,minimum size=4.5mm] at (v1\i) {};
\node[paper reach] at (v06) {$r$};
\node[paper reach] at (v17) {$r'$};
\draw[figTeal,line width=3pt] (4,.42) -- (8,.42);
\draw[figTeal,line width=3pt] (5,-1.08) -- (8,-1.08);
\node[paper label,text=figTeal] at (6,.8) {robber gap $J$};
\node[paper label,text=figOrange!80!black] at (2.5,.8) {$I_L$};
\node[paper label,text=figOrange!80!black] at (9.5,.8) {$I_R$};
\draw[paper arc,figOrange,line width=1.5pt] (15,-.2) -- node[right,paper label] {$+1$} (15,-1.3);
\node[paper label,anchor=east] at (-1.2,0) {before};
\node[paper label,anchor=east] at (-1.2,-1.5) {after};
\end{tikzpicture}

\caption{One step of the cyclic sweep with $m=14$ and $Q=2$. Translating $I_L$ by $+1$ while fixing $I_R$ shrinks the robber's gap. The baseline gives cyclic order, and dashed lines indicate wrap-around. Barriers of $Q$ occupied residues block every allowed step between the gaps.}
\label{fig:cyclic-sweep}
\end{figure}

The following lemma formulates the barrier argument on a cyclic digraph.  Its
main case is the two-barrier strategy of
\cite[Proposition~4.6]{Biswas2026}.  Separate treatment of the two boundary
cases extends that strategy to all quotient orders.

\begin{lemma}[Cyclic barrier]\label{lem:cyclic-barrier}
Let $A\subseteq\Z_m$, where $m\geq2$, generate $\Z_m$.  We allow
$0\in A$; the loops arising from $0$ are suppressed and have no effect on
either player's legal moves. Put $Q:=Q(A)$, using the cyclic displacement
notation in Section~\ref{sec:cayley-notation}.
Then $1\leq Q\leq\lfloor m/2\rfloor$ and
\begin{equation}\label{eq:cyclic-barrier}
 c_{1,\infty}\left(\Cay\left(\Z_m,A\right)\right)\leq2Q.
\end{equation}
\end{lemma}

\begin{proof}
Since $A$ generates the nontrivial group $\Z_m$, it contains a nonzero
residue, so $1\leq Q\leq\lfloor m/2\rfloor$.
We first suppose $m\geq2Q+2$ and choose two disjoint cyclic intervals
$I_L,I_R\subseteq\Z_m$, each consisting of $Q$ consecutive residues, such
that the complement has two nonempty interval components.  The cops occupy
the residues of $I_L\cup I_R$, one cop at each residue.  After the robber chooses its initial vertex,
it lies in one of the two complementary gaps (Figure~\ref{fig:cyclic-sweep});
we call this gap $J$ and label the barriers so that, in the positive cyclic
order on $\Z_m$, $I_L$ lies immediately before $J$ and $I_R$ immediately after it.

We choose $a_*\in A$ with $a_*\neq0$ and let $\widetilde a_*$ be its
least-absolute-value integer representative. If $\widetilde a_*>0$, we move
every cop in the left barrier by the $a_*$-arc and keep the right barrier
fixed. If $\widetilde a_*<0$, we do the same with the right barrier. The moving
barrier remains an interval of $Q$ occupied residues and advances
$|a_*|_m$ residues into $J$. If it reaches the robber, a cop lands on the
robber's vertex. More explicitly, if $|J|\leq|a_*|_m$, the translated barrier
covers all of $J$, because $|a_*|_m\leq Q$. Otherwise an uncaptured robber
remains in a gap of $|J|-|a_*|_m$ residues.

Every arc changes the cyclic coordinate by distance at most $Q$.  A single
step cannot pass from one side of an interval of $Q$ consecutive occupied
residues to the other without landing in the interval.  Hence a directed path
avoiding the cops cannot cross either barrier, so the robber remains in the
shrinking gap.  Each round either captures the robber or removes
$|a_*|_m\geq1$ residues from its gap, so the robber is captured after at
most $|J|$ rounds.

If $m=2Q+1$, we occupy every residue except one, using $2Q$ cops.  If the robber
starts at the empty residue $x$, we choose a nonzero $a\in A$; the cop at $x-a$
moves along the $a$-arc to $x$ on the first turn.  If $m=2Q$, then $2Q=m$ and
one cop at every residue suffices.  These cases exhaust all possibilities.
\end{proof}

The cyclic sweep of \cite[Proposition~4.6]{Biswas2026}, together with
Schreier lifting, gives the following bound for arbitrary finite groups.

\begin{proposition}[Cyclic-quotient sweep]\label{prop:cyclic-sweep}
Let $\Gamma^+=\Cay\left(G,S\right)$ be a strongly connected Cayley digraph on $n$ vertices,
where $G$ is an arbitrary finite group.  Suppose that
$\pi:G\to\Z_m$, where $m\geq2$, is a surjective homomorphism, and put
$Q:=Q_\pi(S)$ as in Section~\ref{sec:cayley-notation}.
Then $1\leq Q\leq\lfloor m/2\rfloor$ and
\begin{equation}\label{eq:cyclic-sweep}
 c_{1,\infty}\left(\Gamma^+\right)\leq\frac{2Qn}{m}.
\end{equation}
\end{proposition}

\begin{proof}
Since $S$ generates $G$ and $\pi$ is surjective, $\pi(S)$ generates $\Z_m$.
Let $K=\ker\pi$.  Under the identification $Kx\mapsto\pi\left(x\right)$, the Schreier
digraph $\Sch\left(K\backslash G,S\right)$ is the cyclic Cayley digraph generated by
$\pi\left(S\right)$, with redundant loops suppressed.  Lemma~\ref{lem:cyclic-barrier}
gives its cop number at most $2Q$.  Since $|K|=n/m$,
Proposition~\ref{prop:quotient-lifting} gives \eqref{eq:cyclic-sweep}.
\end{proof}

\begin{corollary}[Character upper bound]\label{cor:character-upper}
Let $\Gamma^+=\Cay\left(G,S\right)$ be a strongly connected abelian Cayley digraph of order
$n\geq2$.  Then
\begin{equation}\label{eq:character-upper}
 c_{1,\infty}\left(\Gamma^+\right)\leq2n\delta\left(\Gamma^+\right).
\end{equation}
\end{corollary}

\begin{proof}
We choose a nontrivial character $\chi$ attaining the minimum in
\eqref{eq:def-delta}, and let its image have order $m$.  Nontriviality gives
$m\geq2$.  We identify this image with $\Z_m$ by sending $j/m+\Z$ to $j$;
under this identification, $\normtor{j/m}=|j|_m/m$.
This gives a surjective homomorphism
$\pi:G\to\Z_m$. If $Q$ is as in Proposition~\ref{prop:cyclic-sweep}, then
\[
 \delta\left(\Gamma^+\right)=\frac Qm.
\]
The result follows immediately from \eqref{eq:cyclic-sweep}.
\end{proof}

\subsection{A band cut from a character}

\begin{figure}[tbp]
\centering

\begin{tikzpicture}[x=.6cm,y=1.15cm]
\foreach \i in {0,...,15} {
 \def\col{figInk!8}
 \ifnum\i<8 \def\col{figTeal!25}\else
  \ifnum\i<10 \def\col{figOrange!30}\fi
  \ifnum\i>13 \def\col{figOrange!30}\fi
 \fi
 \filldraw[fill=\col,draw=white] (\i,0) rectangle (\i+1,1.6);
 \foreach \j in {1,2,3} \fill[figInk!55] (\i+.5,{.4*\j}) circle (1.3pt);
 \node[font=\scriptsize,text=figInk] at (\i+.5,-.25) {\i};
}
\draw[figInk!60,dashed,rounded corners=5pt] (0,.8) -- (-.6,.8) -- (-.6,2.1) -- (16.6,2.1) -- (16.6,.8) -- (16,.8);
\node[paper label,text=figTeal] at (4,-.85) {$A$: $L$ full fibres};
\node[paper label,text=figOrange!80!black] at (12,-.85) {at most $2Q$ boundary fibres};
\node[paper label] at (8,2.55) {character quotient $\chi\left(G\right)\cong\Z_m$};
\end{tikzpicture}

\caption{A character band with $m=16$, $L=8$ and $Q=2$. The set $A$ consists of the teal fibres; its boundary lies in the $2Q$ orange fibres. Dots represent equal-sized fibres.}
\label{fig:character-band}
\end{figure}
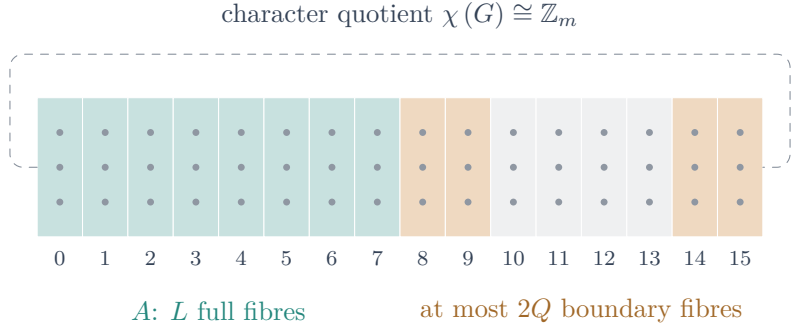

\begin{lemma}[Character band]\label{lem:character-band}
Let $\Gamma^+=\Cay\left(G,S\right)$ be a strongly connected abelian Cayley digraph of order $n\geq2$ and let
$\Gamma=U(\Gamma^+)$ be its underlying undirected graph.  Then
\begin{equation}\label{eq:band-bound}
 \iota_v\left(\Gamma\right)\leq6\delta\left(\Gamma^+\right).
\end{equation}
The constant $6$ may be replaced by $4m/\left(m-1\right)$ if a minimizing character has
odd image order $m$, and by $4$ if it has even image order.
\end{lemma}

\begin{proof}
We choose a minimizing character, identify its image with $\Z_m$, and write
$Q=m\delta\left(\Gamma^+\right)$, using the identification $j/m+\Z\mapsto j$.
The character is nontrivial, so $m\geq2$ and $L:=\lfloor m/2\rfloor\geq1$.
Let $A$ be the full
preimage of $L$ consecutive residues of $\Z_m$
(Figure~\ref{fig:character-band}). Then
$0<|A|=Ln/m\leq n/2$, so $A$ is admissible in the definition of
$\iota_v\left(\Gamma\right)$ (see \eqref{eq:def-hplus}).  Every element of $T=S\cup(-S)$ changes the
quotient coordinate by cyclic distance at most $Q$.  Hence
$\partial_vA$ is contained in the fibres above at most $Q$
residues on each side of the interval.  Therefore
\[
 \frac{|\partial_vA|}{|A|}
 \leq\frac{2Q}{L}
 =\frac{2m}{\lfloor m/2\rfloor}\delta\left(\Gamma^+\right).
\]
The coefficient equals $4$ when $m$ is even and is
$4m/\left(m-1\right)\leq6$ when $m$ is odd.
\end{proof}

\begin{remark}[Optimization of the character parameter]\label{rem:NH}
The character-displacement parameter $\delta$ uses the specific generating
set. For some Cayley graphs, Corollary~\ref{cor:character-upper} gives a
cop-number bound of smaller order than the quotient bound
$c_{1,\infty}\leq2n/N_*$ of \cite[Theorem~1.2]{Biswas2026}; see also
\cite[Section~5.2]{Biswas2026}. For example, let $n\geq5$ and
$\Gamma=\Cay\left(\Z_n,\{\pm1,\pm2\}\right)$.  The character
$x\mapsto x/n$ gives $\delta\left(\Gamma\right)\leq2/n$, and hence
$c_{1,\infty}\left(\Gamma\right)\leq4$.  On the other hand, the parameter
$N_*$ of \cite[Theorem~1.2]{Biswas2026} equals
$\lfloor\left(n-1\right)^{1/2}\rfloor$, so the bound $2n/N_*$ has order
$n^{1/2}$.  An interval of $\lfloor n/2\rfloor$ consecutive vertices has at
most four external boundary vertices, so
$\iota_v\left(\Gamma\right)=O\left(1/n\right)$ and
$N_*\iota_v(\Gamma)=O(n^{-1/2})$.  Thus $N_*$ is not comparable with the
inverse isoperimetric scale $1/\iota_v(\Gamma)$, whereas $\delta(\Gamma)$ is
comparable with $\iota_v(\Gamma)$ by
Theorem~\ref{thm:expansion-spectral-characterization}.
\end{remark}

\section{Abelian Cayley digraphs: expansion, spectrum, and diameter}

Throughout this section, $\Gamma^+=\Cay\left(G,S\right)$ is a strongly
connected abelian Cayley digraph of order $n\geq3$ and out-degree $d$.
We use the notation $\Gamma=U(\Gamma^+)$, $\widehat d$, $h^+$, $\widehat h$,
and $\delta$ fixed in Section~\ref{sec:notation}.

\subsection{The spectral comparison and the Buser inequality}

\begin{proposition}[Character versus spectrum]\label{prop:character-spectrum}
With the notation above,
\begin{equation}\label{eq:character-spectrum}
 8\delta^2\leq\lambda_1\left(\Gamma\right)
 \leq2\pi^2\widehat d\,\delta^2.
\end{equation}
\end{proposition}

\begin{proof}
For $\chi\in\widehat G$, let $f_\chi(g)=\exp\left(2\pi i\chi(g)\right)$ be
the corresponding complex-valued character (Section~\ref{sec:notation}).
Since $f_\chi(g+t)=f_\chi(g)f_\chi(t)$ and $T=-T$, direct calculation gives
$L_{\Gamma}f_\chi=\lambda(\chi)f_\chi$, where
\begin{equation}\label{eq:character-eigenvalue}
 \lambda\left(\chi\right)=
 \sum_{t\in T}\left(1-\cos\left(2\pi\chi\left(t\right)\right)\right).
\end{equation}
The $|G|$ functions $f_\chi$ are pairwise orthogonal, so they form a basis of
eigenvectors for $L_{\Gamma}$.
Since $T$ generates $G$, the equality $\lambda\left(\chi\right)=0$ holds only for the
trivial character.  Hence
\[
 \lambda_1\left(\Gamma\right)
 =\min_{\substack{\chi\in\widehat G\\\chi\neq0}}\lambda\left(\chi\right).
\]
For $0\leq x\leq1/2$,
\begin{equation}\label{eq:cos-bounds}
 8x^2\leq1-\cos\left(2\pi x\right)\leq2\pi^2x^2.
\end{equation}
Both inequalities follow from $1-\cos(2\pi x)=2\sin^2(\pi x)$ and
$2x\leq\sin(\pi x)\leq\pi x$ on $[0,1/2]$.  Moreover,
$1-\cos\left(2\pi\theta\right)$ depends only on $\normtor{\theta}$ for
$\theta\in\R/\Z$, so \eqref{eq:cos-bounds} applies with
$x=\normtor{\theta}$.
For every nontrivial $\chi$, the definition \eqref{eq:def-delta} of $\delta$
gives some $s\in S$ with
$\normtor{\chi\left(s\right)}\geq\delta$.  The corresponding summand in
\eqref{eq:character-eigenvalue} and the left inequality of
\eqref{eq:cos-bounds} show that $\lambda\left(\chi\right)\geq8\delta^2$.  Minimizing over
nontrivial $\chi$ proves the first inequality in
\eqref{eq:character-spectrum}.

Conversely, we choose $\chi$ attaining the minimum in \eqref{eq:def-delta}.  Since
$T=S\cup\left(-S\right)$,
$\normtor{\chi\left(t\right)}\leq\delta$ for every $t\in T$.  The right inequality in
\eqref{eq:cos-bounds} gives
$\lambda\left(\chi\right)\leq2\pi^2\widehat d\delta^2$, and
$\lambda_1\left(\Gamma\right)\leq\lambda\left(\chi\right)$.
\end{proof}

Since $\Gamma$ is an undirected abelian Cayley graph, the Buser inequality of
Klartag--Kozma--Ralli--Tetali gives
\begin{equation}\label{eq:KKRT-buser}
 \lambda_1\left(\Gamma\right)\leq16h_E\left(\Gamma\right)^2
\end{equation}
\cite[Section~1.1 and Theorems~1.1 and~2.3]{KlartagKozmaRalliTetali}.
Here $\lambda_1$ is the spectral gap of the unnormalized Laplacian
$L_{\Gamma}=D_{\Gamma}-A_{\Gamma}$, and $h_E$ counts crossing edges once.

\begin{proposition}[Buser comparison]\label{prop:buser-comparison}
For the abelian Cayley graph $\Gamma$,
\begin{equation}\label{eq:buser-delta}
 \delta\leq\sqrt2\,\widehat d\,\widehat h
 \leq2\sqrt2\,d\,\widehat h.
\end{equation}
\end{proposition}

\begin{proof}
Every vertex in the external vertex boundary of a set is incident with at most
$\widehat d$ crossing edges.  Hence
\[
 h_E\left(\Gamma\right)\leq\widehat d\,\widehat h.
\]
Applying \eqref{eq:KKRT-buser} and then the left inequality of
\eqref{eq:character-spectrum} gives
\[
 8\delta^2\leq\lambda_1\left(\Gamma\right)
 \leq16\widehat d^2\widehat h^2.
\]
Taking square roots proves the first inequality in
\eqref{eq:buser-delta}; the second follows from $\widehat d\leq2d$.
\end{proof}

\begin{theorem}[Expansion and spectral characterization]
\label{thm:expansion-spectral-characterization}
With the notation of this section,
\begin{gather}
 h^+\leq\widehat h\leq\left(d+1\right)h^+,
 \label{eq:h-comparison-summary}\\
 \widehat h\leq6\delta,
 \qquad
 \delta\leq\sqrt2\widehat d\,\widehat h,
 \label{eq:delta-h-summary}\\
 8\delta^2\leq\lambda_1\left(\Gamma\right)
 \leq2\pi^2\widehat d\delta^2,
 \label{eq:spectral-summary}
\end{gather}
and
\begin{equation}\label{eq:cop-explicit-chain}
 \frac{nh^+}{4\left(d+1\right)}
 <c_{1,\infty}\left(\Gamma^+\right)
 \leq2n\delta
 \leq2\sqrt2\widehat d\,n\widehat h
 \leq2\sqrt2\widehat d\left(d+1\right)nh^+.
\end{equation}
Consequently
\[
 c_{1,\infty}\left(\Gamma^+\right)=\Theta_d\left(nh^+\right)
 =\Theta_d\left(n\widehat h\right)
 =\Theta_d\left(n\delta\right)
 =\Theta_d\left(n\sqrt{\lambda_1\left(\Gamma\right)}\right).
\]
\end{theorem}

\begin{proof}
Equation \eqref{eq:h-comparison-summary} is
Lemma~\ref{lem:in-out-comparison}, applied with $D^-=d$
(Remark~\ref{rem:cayley-eulerian}), and
\eqref{eq:delta-h-summary} follows from Lemma~\ref{lem:character-band} and
Proposition~\ref{prop:buser-comparison}.  Equation
\eqref{eq:spectral-summary} is Proposition~\ref{prop:character-spectrum}.
The lower bound in \eqref{eq:cop-explicit-chain} is
Theorem~\ref{thm:directed-lower}, while the upper bound is
Corollary~\ref{cor:character-upper}.  The remaining inequalities follow from
\eqref{eq:delta-h-summary} and \eqref{eq:h-comparison-summary}.
\end{proof}

\subsection{Diameter and bounded doubling}

For fixed degree, the comparison of vertex expansion with reciprocal diameter
follows from \cite[Proposition~4.7 and Corollary~1.6]{BreuillardTointon},
using the degree-dependent comparison between edge and vertex expansion.
We give an elementary proof with the explicit upper bound $24d/R$.
It uses the fact that a Cayley graph of an abelian group generated by $d$
elements has a doubling constant depending only on $d$.

For $r\geq0$, let $B\left(r\right)$ be the radius-$r$ ball about $0$ in
$\Gamma$.

\begin{lemma}[Ball doubling]\label{lem:ball-doubling}
For every integer $r\geq1$,
\begin{equation}\label{eq:ball-doubling}
 |B\left(2r\right)|\leq5^d|B\left(r\right)|.
\end{equation}
\end{lemma}

\begin{proof}
We enumerate $S=\{s_1,\ldots,s_d\}$ and consider the surjective homomorphism
\[
 \phi:\Z^d\longrightarrow G,
 \qquad
 \phi\left(z_1,\ldots,z_d\right)=\sum_{i=1}^d z_i s_i.
\]
If
\[
 C_r:=\{z\in\Z^d:\|z\|_1\leq r\},
\]
then $B\left(r\right)=\phi\left(C_r\right)$.

We choose a maximal set $X\subseteq C_{2r}$ whose distinct elements have
$\ell^1$-distance greater than $r$.  Maximality implies
\[
 C_{2r}\subseteq\bigcup_{x\in X}\left(x+C_r\right).
\]
The interiors of the continuous $\ell^1$-balls
$x+\left(r/2\right)B_1^d$, $x\in X$, are pairwise disjoint, and all are contained in
$\left(5r/2\right)B_1^d$.  Comparing volumes gives $|X|\leq5^d$.  Applying $\phi$ to the
cover above shows that $B\left(2r\right)$ is covered by at most $5^d$ translates of
$B\left(r\right)$, proving \eqref{eq:ball-doubling}.
\end{proof}

\begin{proposition}[Expansion versus diameter]\label{prop:expansion-diameter}
Let $R=\diam\left(\Gamma\right)$.  Then
\begin{equation}\label{eq:diameter-expansion-explicit}
 \frac{2}{2R+1}\leq\widehat h
 \leq\frac{24d}{R}.
\end{equation}
In particular, $\widehat h=\Theta_d\left(1/R\right)$.
\end{proposition}

\begin{proof}
The graph $\Gamma$ is connected and has $n\geq3$ vertices, so $R\geq1$
and $\widehat h>0$.  The lower bound is \eqref{eq:babai-szegedy-half} in
Lemma~\ref{lem:babai-szegedy}, applied to the vertex-transitive graph
$\Gamma$.

For the upper bound, we first assume $R\geq5$ and set
\[
 t:=\left\lfloor\frac{R-1}{4}\right\rfloor.
\]
Then $t\geq1$, $4t<R$, and $t\geq R/8$.  We choose $g\in G$ with
$\operatorname{dist}_{\Gamma}\left(0,g\right)=R$.  The balls $B\left(2t\right)$ and
$g+B\left(2t\right)$ are disjoint: an intersection would imply
$g\in B\left(2t\right)-B\left(2t\right)\subseteq B\left(4t\right)$, contrary to $4t<R$.  Hence
\begin{equation}\label{eq:half-ball}
 |B\left(2t\right)|\leq n/2.
\end{equation}
For every $j\in\{t,\ldots,2t-1\}$, the ball $B(j)$ contains $0$ and
satisfies $|B(j)|\leq|B(2t)|\leq n/2$, so it is admissible in the
definition of $\widehat h$, and $|\partial_v B(j)|\geq\widehat h|B(j)|$.
Since $\partial_v B(j)\subseteq B(j+1)\setminus B(j)$, it follows that
$|B(j+1)|\geq(1+\widehat h)|B(j)|$. Multiplying these inequalities over $j$
and applying Lemma~\ref{lem:ball-doubling} with $r=t$ gives
\[
 (1+\widehat h)^t\leq
 \prod_{j=t}^{2t-1}\frac{|B(j+1)|}{|B(j)|}
 =\frac{|B(2t)|}{|B(t)|}\leq5^d.
\]
Thus $t\log\left(1+\widehat h\right)\leq d\log5$.  The estimate
$0<\widehat h\leq2$, which is \eqref{eq:hplus-at-most-two} applied to
$\Gamma$ viewed as a symmetric digraph, and concavity of the
logarithm give
\[
 \log\left(1+\widehat h\right)\geq\frac{\log3}{2}\widehat h.
\]
Consequently
\[
 \widehat h\leq\frac{2d\log5}{t\log3}
 \leq\frac{16\log5}{\log3}\frac dR
 <\frac{24d}{R},
\]
since $25<27$.

If $R\leq4$, the same displayed bound follows from the estimate
$\widehat h\leq2$, since $24d/R\geq6$.
\end{proof}

Combining the preceding results gives the following comparisons.

\begin{theorem}[Abelian characterization]\label{thm:abelian-characterization}
Let $\Gamma^+=\Cay\left(G,S\right)$ be a strongly connected abelian Cayley digraph on
$n\geq3$ vertices, with $|S|=d$.  Then
\[
 \begin{split}
 c_{1,\infty}\left(\Gamma^+\right)
 &=\Theta_d\left(n\iota_v^+\left(\Gamma^+\right)\right)
 =\Theta_d\left(n\iota_v\left(\Gamma\right)\right)
 =\Theta_d\left(n\delta\left(\Gamma^+\right)\right)\\
 &=\Theta_d\left(n\sqrt{\lambda_1\left(\Gamma\right)}\right)
 =\Theta_d\left(\frac{n}{\diam\left(\Gamma\right)}\right).
 \end{split}
\]
\end{theorem}

\begin{proof}
The first four comparisons are Theorem
\ref{thm:expansion-spectral-characterization}; the final one follows from
Proposition~\ref{prop:expansion-diameter}.
\end{proof}

\begin{corollary}[Quantitative cop-number bounds in terms of diameter]
\label{cor:cop-diameter}
Let $\Gamma^+=\Cay\left(G,S\right)$ be a strongly connected abelian Cayley
digraph of order $n\geq3$ and out-degree $d$.  Put
$R=\diam\left(\Gamma\right)$.  Then
\[
 \frac{n}{2\left(d+1\right)^2\left(2R+1\right)}
 <c_{1,\infty}\left(\Gamma^+\right)
 \leq96\sqrt2\,d^2\frac nR.
\]
\end{corollary}

\begin{proof}
Lemma~\ref{lem:in-out-comparison}, applied with $D^-=d$
(Remark~\ref{rem:cayley-eulerian}), and
Proposition~\ref{prop:expansion-diameter} give
\[
 h^+\geq\frac{\widehat h}{d+1}
 \geq\frac{2}{\left(d+1\right)\left(2R+1\right)}.
\]
We substitute this into Theorem~\ref{thm:directed-lower} for the lower bound.
For the upper bound, \eqref{eq:cop-explicit-chain} gives
$c_{1,\infty}\left(\Gamma^+\right)\leq2\sqrt2\,\widehat d\,n\widehat h$;
we use $\widehat d\leq2d$ and $\widehat h\leq24d/R$ from
Proposition~\ref{prop:expansion-diameter}.
\end{proof}

\begin{remark}[Dependence on the degree]\label{rem:degree-dependence}
The displayed bounds give polynomial dependence on $d$ in every comparison
in Theorem~\ref{thm:abelian-characterization}.
In particular, using $\widehat d\leq2d$ in
\eqref{eq:cop-explicit-chain}, the ratio of the upper and lower
coefficients of $nh^+$ is at most
$16\sqrt2\,d\left(d+1\right)^2=O\left(d^3\right)$.
\end{remark}

\section{The undirected case and width parameters}
\label{sec:width}

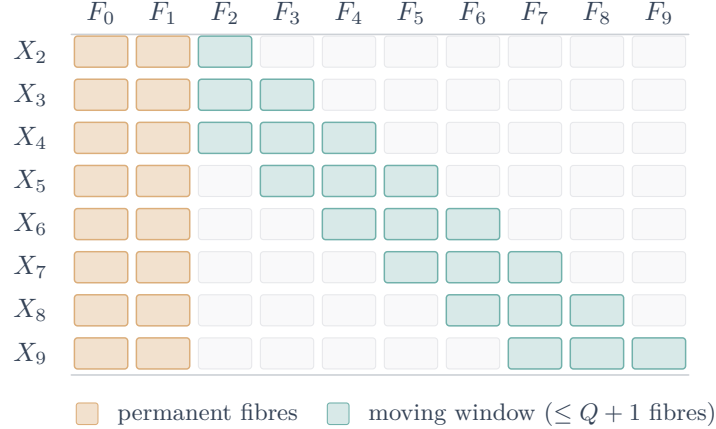
\begin{figure}[tbp]
\centering

\begin{tikzpicture}[x=.82cm,y=.57cm]
\foreach \i in {0,...,9}
 \node[font=\small,text=figInk] at (\i,.9) {$F_{\i}$};
\draw[figInk!25,line width=.6pt] (-.48,.4) -- (9.48,.4);
\foreach \j in {2,...,9} {
 \pgfmathsetmacro{\yy}{2-\j}
 \node[paper label,anchor=east] at (-.7,\yy) {$X_{\j}$};
 \foreach \i in {0,...,9}
  \filldraw[fill=figInk!3,draw=figInk!12,rounded corners=1.2pt,line width=.4pt]
   (\i-.43,\yy-.36) rectangle (\i+.43,\yy+.36);
 \foreach \i in {0,1}
  \filldraw[fill=figOrange!24,draw=figOrange!65,rounded corners=1.2pt,line width=.55pt]
   (\i-.43,\yy-.36) rectangle (\i+.43,\yy+.36);
 \pgfmathtruncatemacro{\start}{max(2,\j-2)}
 \foreach \i in {\start,...,\j}
  \filldraw[fill=figTeal!18,draw=figTeal!65,rounded corners=1.2pt,line width=.55pt]
   (\i-.43,\yy-.36) rectangle (\i+.43,\yy+.36);
}
\draw[figInk!25,line width=.6pt] (-.48,-7.5) -- (9.48,-7.5);
\filldraw[fill=figOrange!24,draw=figOrange!65,rounded corners=1pt]
 (-.4,-8.65) rectangle (-.05,-8.15);
\node[paper label,font=\footnotesize,anchor=west] at (.1,-8.4) {permanent fibres};
\filldraw[fill=figTeal!18,draw=figTeal!65,rounded corners=1pt]
 (3.65,-8.65) rectangle (4,-8.15);
\node[paper label,font=\footnotesize,anchor=west] at (4.15,-8.4)
 {moving window ($\leq Q+1$ fibres)};
\end{tikzpicture}

\caption{Path-decomposition bags for $m=10$ and $Q=2$. A coloured cell indicates that $F_i\subseteq X_j$. The orange fibres belong to every bag; the teal window contains at most $Q+1$ fibres. Each bag has at most $(2Q+1)n/m$ vertices.}
\label{fig:path-bags}
\end{figure}

In this section $\Gamma=\Cay\left(G,S\right)$ is a connected undirected abelian
Cayley graph of order $n\geq3$, so $S=-S$.  The character quotient also controls pathwidth.

\begin{proposition}[A cyclic-quotient path decomposition]
\label{prop:pathwidth-quotient}
Suppose $\pi:G\to\Z_m$ is a surjective homomorphism and every $s\in S$ has
cyclic displacement at most $Q$, where $1\leq Q\leq\lfloor m/2\rfloor$.  Then
\begin{equation}\label{eq:pathwidth-quotient}
 \pw\left(\Gamma\right)+1\leq\left(2Q+1\right)\frac nm
 \leq3Q\frac nm.
\end{equation}
Consequently
\begin{equation}\label{eq:pathwidth-delta}
 \pw\left(\Gamma\right)+1\leq3n\delta\left(\Gamma\right).
\end{equation}
\end{proposition}

\begin{proof}
We identify the quotient residues with $0,1,\ldots,m-1$, and let
$F_i=\pi^{-1}\left(i\right)$.  Every fibre has size $n/m$.  We keep the first $Q$ fibres in
every bag (Figure~\ref{fig:path-bags}). For $j=Q,Q+1,\ldots,m-1$, we define
\[
 X_j:=
 \bigcup_{i=0}^{Q-1}F_i
 \;\cup\!
 \bigcup_{i=\max\{Q,j-Q\}}^jF_i.
\]
Every vertex occurs in at least one bag.  Vertices in the first $Q$ fibres
occur in all bags, while a vertex of $F_i$ with $i\geq Q$ occurs in a
consecutive interval of bags.

We now check that every edge is covered.  An edge with an endpoint in one of
$F_0,\ldots,F_{Q-1}$ is covered because those fibres are included in every bag. If both
endpoints lie in the same fibre $F_i$ with $i\geq Q$, any bag containing that fibre
covers the edge.  If the two distinct
endpoint residues $u<v$ lie in $\{Q,\ldots,m-1\}$, their wrap-around distance
is $m-v+u\geq Q+1$.  Hence cyclic distance at most $Q$ implies
$v-u\leq Q$, and both fibres occur in $X_v$.  Thus $\left(X_j\right)$ is a path
decomposition.

Each bag contains the first $Q$ fibres and at most $Q+1$ additional fibres
in the consecutive interval defining the bag, so its size is at most
$\left(2Q+1\right)n/m$. This proves
\eqref{eq:pathwidth-quotient}.  We apply this bound to a character attaining $\delta$;
then $Q/m=\delta$, and $2Q+1\leq3Q$.
\end{proof}

Alon--Mehrabian proved that every connected undirected graph $\Gamma$ satisfies
\begin{equation}\label{eq:treewidth-cop-upper}
 \frac{\tw(\Gamma)+1}{\Delta(\Gamma)+1}
 \leq c_{1,\infty}(\Gamma)\leq\tw(\Gamma)+1
\end{equation}
\cite[Theorem~3.3]{AlonMehrabian2015}, where $\Delta(\Gamma)$ is its maximum
degree. Thus $c_{1,\infty}(\Gamma)=\Theta_\Delta(\tw(\Gamma)+1)$ for every
connected graph of maximum degree at most $\Delta$, and the treewidth
comparison in Corollary~\ref{cor:width-characterization} below is a special
case. Since $\tw\leq\pw$, the upper bound in \eqref{eq:treewidth-cop-upper}
also gives $c_{1,\infty}(\Gamma)\leq\pw(\Gamma)+1$. For abelian Cayley
graphs, Proposition~\ref{prop:pathwidth-quotient} supplies the reverse
bound for pathwidth, which we compare with expansion, character
displacement, spectrum, and diameter.
Combining the upper bound in \eqref{eq:treewidth-cop-upper} with
Theorem~\ref{thm:directed-lower}, applied to $\Gamma$ viewed as a symmetric
digraph (so that $D^+=d$ and $\iota_v^+(\Gamma)=\iota_v(\Gamma)$), the
inequality $\tw\leq\pw$ from Section~\ref{sec:notation}, and
Proposition~\ref{prop:pathwidth-quotient} gives the explicit chain
\eqref{eq:width-explicit-chain} below. The remaining comparisons follow
from Theorem~\ref{thm:abelian-characterization}, applied to $\Gamma$ viewed
as a symmetric digraph, whose underlying undirected graph is $\Gamma$.

\begin{corollary}[Width characterization]\label{cor:width-characterization}
Let $\Gamma$ be a connected undirected abelian Cayley graph of order $n\geq3$ and
degree $d$.  Then
\[
 \begin{split}
 c_{1,\infty}\left(\Gamma\right)
 &=\Theta_d\left(n\iota_v\left(\Gamma\right)\right)
 =\Theta_d\left(n\delta\left(\Gamma\right)\right)
 =\Theta_d\left(n\sqrt{\lambda_1\left(\Gamma\right)}\right)\\
 &=\Theta_d\left(\frac n{\diam\left(\Gamma\right)}\right)
 =\Theta_d\left(\tw\left(\Gamma\right)+1\right)
 =\Theta_d\left(\pw\left(\Gamma\right)+1\right).
 \end{split}
\]
More explicitly,
\begin{equation}\label{eq:width-explicit-chain}
 \frac{n\iota_v\left(\Gamma\right)}{4\left(d+1\right)}
 <c_{1,\infty}\left(\Gamma\right)
 \leq\tw\left(\Gamma\right)+1
 \leq\pw\left(\Gamma\right)+1
 \leq3n\delta\left(\Gamma\right).
\end{equation}
\end{corollary}

\begin{remark}
Mehrabian's estimate \eqref{eq:mehrabian-lower-intro} improves the first
constant here from $1/\left(4\left(d+1\right)\right)$ to $1/\left(4d\right)$,
with a non-strict inequality. Both bounds have the same order for fixed $d$.
\end{remark}

\section{Applications}

\subsection{Directed tori and sharpness of the bounded-degree exponent}
\label{sec:directed-tori}

For integers $m\geq3$ and $k\geq1$, define the positively oriented torus
\[
 \vec T_{m,k}:=\Cay\left(\Z_m^k,\{e_1,\ldots,e_k\}\right).
\]
It is strongly connected, has order $n=m^k$ and out-degree $k$.
Its underlying undirected graph $U(\vec T_{m,k})$ is $C_m^{\square k}$.

For the undirected torus $C_m^{\square k}$, the order of the cop number is
known.  When $m$ is even, Mehrabian's
Cartesian-product theorem gives
\[
 \frac{m^{k-1}}{2k^2}
 \leq c_{1,\infty}\left(C_m^{\square k}\right)
 \leq2m^{k-1}
\]
\cite[Theorem~5.1(c)]{Mehrabian2012}.  In dimension two,
Kinnersley--Townsend sharpened this to
\[
 2m-24\leq c_{1,\infty}\left(C_m\square C_m\right)\leq2m
 \qquad\left(m\geq18\right)
\]
\cite[Theorem~2.7]{KinnersleyTownsend}.  Theorem~\ref{thm:directed-tori}
below determines the order of $c_{1,\infty}$ for positively oriented tori.
It shows that the exponent $1-1/k$ in the
upper bound of \cite[Theorem~1.3]{Biswas2026} for abelian Cayley digraphs of
out-degree $k$ is attained.

\begin{lemma}[Geometry of directed tori]\label{lem:torus-geometry}
Let $m\geq3$, $k\geq1$, and $L=\lfloor m/2\rfloor$.
Then
\[
 \delta\left(\vec T_{m,k}\right)=\frac1m,
 \qquad \diam\left(U(\vec T_{m,k})\right)=kL,
\]
and
\[
 \frac{2}{\left(k+1\right)\left(2kL+1\right)}
 \leq\iota_v^+\left(\vec T_{m,k}\right)
 \leq\frac1L.
\]
In particular, $\iota_v^+\left(\vec T_{m,k}\right)=\Theta_k\left(1/m\right)$.
\end{lemma}

\begin{proof}
The characters of $\Z_m^k$ are $\chi_\ell(x)=\ell\cdot x/m+\Z$ for
$\ell\in\Z_m^k$, and $\chi_\ell(e_i)=\ell_i/m+\Z$.  If $\ell\neq0$, some
$\ell_i\neq0$, so $\normtor{\chi_\ell(e_i)}\geq1/m$.  On the other hand,
$\chi_{e_1}$ has maximum generator norm exactly $1/m$.
This proves the character formula, by \eqref{eq:def-delta}.
Distance in the undirected torus is the sum of the cyclic coordinate
distances, each at most $L$, and the vertex with every coordinate equal
to $L$ attains distance $kL$ from the origin.
Lemma~\ref{lem:babai-szegedy} therefore gives
\[
 \iota_v\left(U(\vec T_{m,k})\right)\geq\frac{2}{2kL+1}.
\]
Lemma~\ref{lem:in-out-comparison}, with $D^-=k$, gives the directed
expansion lower bound.  For the upper bound, we take
\[
 A=\{x\in\Z_m^k:0\leq x_1<L\}.
\]
This set has $Lm^{k-1}\leq m^k/2$ vertices.  Only an $e_1$-move can
leave it, and its out-boundary is the layer $x_1=L$, of size $m^{k-1}$.
Thus the boundary ratio is $1/L$.
\end{proof}

\begin{theorem}[Directed tori]\label{thm:directed-tori}
For every fixed $k\geq1$, uniformly for $m\geq3$ and $n=m^k$,
\begin{equation}\label{eq:directed-torus-order}
 c_{1,\infty}\left(\vec T_{m,k}\right)
 =\Theta_k\left(m^{k-1}\right)
 =\Theta_k\left(n^{1-1/k}\right).
\end{equation}
More explicitly, if $L=\lfloor m/2\rfloor$, then
\begin{equation}\label{eq:directed-torus-explicit}
 \frac{m^k}{2\left(k+1\right)^2\left(2kL+1\right)}
 <c_{1,\infty}\left(\vec T_{m,k}\right)
 \leq2m^{k-1}.
\end{equation}
For $k=1$, the exact value is $c_{1,\infty}\left(\vec T_{m,1}\right)=2$.
\end{theorem}

\begin{proof}
Lemma~\ref{lem:torus-geometry} gives
$\delta\left(\vec T_{m,k}\right)=1/m$, so
Corollary~\ref{cor:character-upper} yields the upper bound.
Theorem~\ref{thm:directed-lower}, together with the
expansion lower bound in Lemma~\ref{lem:torus-geometry}, yields
\[
 c_{1,\infty}\left(\vec T_{m,k}\right)
 >\frac{m^k}{2\left(k+1\right)^2\left(2kL+1\right)}.
\]
For $k=1$, two cops suffice by Lemma~\ref{lem:cyclic-barrier}.
Against one cop, the robber starts two steps ahead and moves one step
precisely when the cop moves, restoring this distance after every round.
This proves the exact value; see also \cite[Remark~4.8]{Biswas2026}.
\end{proof}

The undirected assertion below is known from
\cite[Theorem~5.1(c)]{Mehrabian2012}; see also
\cite[Theorem~1.3 and Section~5.2]{Biswas2026}.

\begin{corollary}[Sharpness of the degree exponents]
\label{cor:degree-exponent-sharpness}
For every fixed integer $D\geq1$, there is a family of strongly connected
oriented abelian Cayley digraphs of out-degree $D$, with orders $n$ tending
to infinity, such that
\[
 c_{1,\infty}\left(\Gamma^+\right)=\Theta_D\left(n^{1-1/D}\right).
\]
For every fixed even integer $D\geq2$, there is a family of connected
undirected abelian Cayley graphs of degree $D$ satisfying
\[
 c_{1,\infty}\left(\Gamma\right)=\Theta_D\left(n^{1-2/D}\right).
\]
\end{corollary}

\begin{proof}
For the directed assertion, we take $\Gamma^+=\vec T_{m,D}$ and let $m$ tend
to infinity.  Its arcs add a positive coordinate vector, and $m\geq3$
ensures that no arc has its reverse, so the digraph is oriented.
Theorem~\ref{thm:directed-tori} gives the claimed order.
For the undirected assertion, we put $k=D/2$ and take
$\Gamma=U(\vec T_{m,k})$, of order $n=m^k$ and degree $2k=D$.
Lemma~\ref{lem:torus-geometry} and the identity
$\delta\left(U(\Gamma^+)\right)=\delta\left(\Gamma^+\right)$ of
Section~\ref{sec:cayley-notation} give $\delta\left(\Gamma\right)=1/m$.
Corollary~\ref{cor:width-characterization} then gives
$c_{1,\infty}\left(\Gamma\right)=\Theta_D\left(n\delta(\Gamma)\right)
=\Theta_D\left(n/m\right)$.
\end{proof}

The directed assertion attains the exponent in the bounded-out-degree
upper bound of \cite[Theorem~1.3]{Biswas2026}.

\subsection{The directed Heisenberg graph}
\label{sec:heisenberg}

For an odd prime $p$, let $H_p$ be $\F_p^3$ with multiplication
\begin{equation}\label{eq:heisenberg-law}
 (a,b,c)(a',b',c')=(a+a',b+b',c+c'+ab').
\end{equation}
Write $x=(1,0,0)$, $y=(0,1,0)$, and $z=(0,0,1)$, and define
\[
 \mathcal H_p:=\Cay(H_p,\{x,y\}),\qquad\text{so that}\qquad
 U(\mathcal H_p)=\Cay(H_p,\{x,x^{-1},y,y^{-1}\}).
\]
The element $z$ is central, and $[x,y]=z$.
Hence $x$ and $y$ generate $H_p$, and $\mathcal H_p$ is strongly connected
of order $p^3$ and out-degree $2$.

\begin{lemma}[Diameter and expansion of the Heisenberg graph]
\label{lem:heisenberg-geometry}
The underlying undirected graph $U(\mathcal H_p)$ satisfies
\begin{equation}\label{eq:heisenberg-diameter}
 \frac{p-1}{2}\leq\diam\left(U(\mathcal H_p)\right)\leq2p,
\end{equation}
and
\begin{equation}\label{eq:heisenberg-expansion}
 \frac{2}{4p+1}\leq
 \iota_v\left(U(\mathcal H_p)\right)
 \leq\frac{4}{p-1}.
\end{equation}
Moreover,
\begin{equation}\label{eq:heisenberg-directed-expansion}
 \frac{2}{3\left(4p+1\right)}\leq
 \iota_v^+\left(\mathcal H_p\right)
 \leq\frac{2}{p-1}.
\end{equation}
\end{lemma}

\begin{proof}
Right multiplication by $x^{\pm1}$ changes the first coordinate by
$\pm1$, while right multiplication by $y^{\pm1}$ leaves it unchanged.  Thus
reaching $\left(\left(p-1\right)/2,0,0\right)$ from the identity requires at least $\left(p-1\right)/2$
steps, proving the diameter lower bound.

For the upper bound, we choose least-absolute-value integer representatives
$\alpha,\beta$ of $a,b\in\F_p$.  Since
\[
 x^\alpha y^\beta=\left(a,b,ab\right),
\]
it remains to correct the central coordinate.  We choose a least-absolute-value
representative $\tau$ of $c-ab$.  The identity
\[
 [x,y^\tau]=z^\tau
\]
expresses this correction with length at most $2+2|\tau|\leq p+1$.
Since $|\alpha|+|\beta|\leq p-1$, every group element has word length at most
$2p$, proving the diameter upper bound.

The lower bound in \eqref{eq:heisenberg-expansion} follows from
Lemma~\ref{lem:babai-szegedy} and the diameter upper bound:
\[
 \iota_v\left(U(\mathcal H_p)\right)
 \geq\frac{2}{2\diam\left(U(\mathcal H_p)\right)+1}
 \geq\frac{2}{4p+1}.
\]
For the upper bound, we set $L=\left(p-1\right)/2$ and take the slab
\[
 A:=\{\left(a,b,c\right):a\in\{0,1,\ldots,L-1\}\}.
\]
It has $Lp^2$ vertices.  The $y^{\pm1}$ moves preserve the first coordinate,
while the $x^{\pm1}$ moves produce at most the two adjacent layers.  Hence
$|\partial_v A|\leq2p^2$, giving the upper bound in
\eqref{eq:heisenberg-expansion}.

Lemma~\ref{lem:in-out-comparison} with $D^-=2$ gives
\[
 \iota_v^+\left(\mathcal H_p\right)
 \geq\frac13\iota_v\left(U(\mathcal H_p)\right),
\]
which proves the lower bound in
\eqref{eq:heisenberg-directed-expansion}.  For the upper bound, we use the same
slab.  Of the two directed generators, only $x$ can leave it, and it produces
one layer.  Therefore
$|\partial_v^+A|=p^2$ and
$\iota_v^+\left(\mathcal H_p\right)\leq1/L=2/\left(p-1\right)$.
\end{proof}

\begin{theorem}[Directed Heisenberg graphs]\label{thm:heisenberg-cops}
For every odd prime $p$,
\begin{equation}\label{eq:heisenberg-cops}
 \frac{p^3}{18\left(4p+1\right)}
 <c_{1,\infty}\left(\mathcal H_p\right)
 \leq2p^2.
\end{equation}
In particular,
\[
 c_{1,\infty}\left(\mathcal H_p\right)=\Theta\left(p^2\right)=\Theta\left(|H_p|^{2/3}\right).
\]
\end{theorem}

\begin{proof}
The projection
\[
 \pi:H_p\longrightarrow\Z_p,
 \qquad \pi\left(a,b,c\right)=a,
\]
is a surjective homomorphism with fibres of size $p^2$; it sends $x$ to $1$
and $y$ to $0$. Proposition~\ref{prop:cyclic-sweep} with $Q=1$ gives
$c_{1,\infty}\left(\mathcal H_p\right)\leq2p^2$.

For the lower bound, Lemma~\ref{lem:heisenberg-geometry} gives
\[
 \iota_v^+\left(\mathcal H_p\right)\geq\frac{2}{3\left(4p+1\right)}.
\]
The maximum out-degree is $2$, so Theorem~\ref{thm:directed-lower} yields
\[
 c_{1,\infty}\left(\mathcal H_p\right)
 >\frac{p^3}{12}\cdot\frac{2}{3\left(4p+1\right)}
 =\frac{p^3}{18\left(4p+1\right)}.
\]
\end{proof}

\subsection{Higher-dimensional Heisenberg groups over cyclic rings}
\label{sec:heisenberg-general}

The preceding argument uses only addition and multiplication in $\F_p$.  It
therefore extends to the rings $\Z_m$ with composite $m$ and to
higher-dimensional Heisenberg groups.

For $m\geq3$ and $r\geq1$, define the higher-dimensional Heisenberg group
and its directed Cayley graph by
\[
 H_{m,r}:=(\Z_m)^r\times(\Z_m)^r\times\Z_m,
 \qquad
 (a,b,c)(a',b',c')=(a+a',b+b',c+c'+a\cdot b'),
\]
\[
 x_i=(e_i,0,0),\quad y_i=(0,e_i,0)\quad(1\leq i\leq r),
 \qquad
 \mathcal H_{m,r}:=\Cay(H_{m,r},\{x_1,\ldots,x_r,y_1,\ldots,y_r\}).
\]
\begin{theorem}\label{thm:heisenberg-general}
Fix $r\geq1$ and $m\geq3$.
Then $|H_{m,r}|=m^{2r+1}$, the digraph has out-degree $2r$, and
\[
 \frac{m^{2r+1}}{2\left(2r+1\right)^2\left(4\left(r+1\right)\lfloor m/2\rfloor+5\right)}
 < c_{1,\infty}\left(\mathcal H_{m,r}\right)
 \leq 2m^{2r}.
\]
In particular, for fixed $r$,
\[
 c_{1,\infty}\left(\mathcal H_{m,r}\right)
 =\Theta_r\left(m^{2r}\right)
 =\Theta_r\left(|H_{m,r}|^{1-1/\left(2r+1\right)}\right).
\]
\end{theorem}

\begin{proof}
Bilinearity of the dot product proves associativity: the central coordinate
in either product of three factors is
$c+c'+c''+a\cdot b'+a\cdot b''+a'\cdot b''$.
The identity is $\left(0,0,0\right)$ and
$\left(a,b,c\right)^{-1}=\left(-a,-b,-c+a\cdot b\right)$.
Thus $H_{m,r}$ is a group of the stated order.
We write $z=\left(0,0,1\right)$, which is central.  Direct calculation gives
$[x_i,y_j]=z^{\delta_{ij}}$, where $\delta_{ij}$ is the Kronecker delta, so the displayed generators generate the
group and the Cayley digraph is strongly connected.  They are distinct and
nonidentity, so the out-degree is $2r$.

To bound the diameter, we set $L=\lfloor m/2\rfloor$ and choose integer representatives
$\alpha_i,\beta_i$ of $a_i,b_i$ with absolute values at most $L$.
Such representatives exist for both odd and even $m$; all word lengths
below are taken in $U(\mathcal H_{m,r})$, where inverse generators are
available.
Then
\[
 \prod_{i=1}^r x_i^{\alpha_i}y_i^{\beta_i}=\left(a,b,a\cdot b\right).
\]
We choose $\tau$ representing $c-a\cdot b$ with $|\tau|\leq L$.
The identity $[x_1,y_1^\tau]=z^\tau$ gives a word for the central correction
of length at most $2+2L$.  Appending this word adds $\tau$ to the central
coordinate and leaves $a$ and $b$ unchanged, giving $\left(a,b,c\right)$.
Hence $U(\mathcal H_{m,r})$ has diameter
\[
 R\leq 2rL+2+2L=2\left(r+1\right)L+2.
\]
The projection $\left(a,b,c\right)\mapsto a_1$ is a surjective homomorphism to $\Z_m$;
it sends $x_1$ to $1$ and all other generators to $0$.
Proposition~\ref{prop:cyclic-sweep} gives the upper bound $2m^{2r}$.

For the lower bound, \eqref{eq:babai-szegedy-half} in
Lemma~\ref{lem:babai-szegedy}, applied to $U(\mathcal H_{m,r})$, and Lemma~\ref{lem:in-out-comparison} with $D^-=2r$ imply
\[
 h^+:=\iota_v^+\left(\mathcal H_{m,r}\right)
 \geq\frac{2}{\left(2r+1\right)\left(2R+1\right)}.
\]
Theorem~\ref{thm:directed-lower} gives
\[
 c_{1,\infty}\left(\mathcal H_{m,r}\right)
 >\frac{m^{2r+1}h^+}{4\left(2r+1\right)}
 \geq\frac{m^{2r+1}}{2\left(2r+1\right)^2\left(2R+1\right)}.
\]
We finish by substituting $2R+1\leq4\left(r+1\right)L+5$.
\end{proof}

\section*{Acknowledgements}
Claude Code assisted with write-up and preparation of figures. The author is responsible for the content.

\end{document}